\documentclass[a4,12pt,reqno]{amsart}

\usepackage{amsfonts,amsmath,amsthm}
\usepackage{amssymb,epsfig}
\usepackage{fancybox}
\usepackage[centertags]{amsmath}
\usepackage{graphicx}
\usepackage{a4wide}

\usepackage{cases}
\usepackage[usenames]{color}
\usepackage{enumerate} 

\usepackage{color} 
\definecolor{vert}{rgb}{0,0.6,0}

 \theoremstyle{plain}
 \begingroup
 \newtheorem{thm}{Theorem}[section]
  
 \newtheorem{lem}[thm]{Lemma}
 \newtheorem{prop}[thm]{Proposition}
 \newtheorem{cor}[thm]{Corollary}
 \endgroup

 \theoremstyle{definition}
 \begingroup
 
 \newtheorem{example}[thm]{Example}
 \endgroup

 \theoremstyle{remark}
 \begingroup
 \newtheorem{rem}[thm]{Remark}
 \endgroup

 \numberwithin{equation}{section}

\newcommand{\E}{\mathbb{E}}

\newcommand{\N}{\mathbb{N}}
\newcommand{\bP}{\mathbb{P}}
\newcommand{\R}{\mathbb{R}}

\newcommand{\T}{\mathbb{T}}
\newcommand{\Z}{\mathbb{Z}}

\newcommand{\AC}{{\rm AC\,}}

\newcommand{\bO}{\partial\Omega}
\newcommand{\cO}{\overline\Omega}
\newcommand{\Q}{\mathbb{R}^{n}\times(0,T)}

\newcommand{\cQ}{\mathbb{R}^{n}\times[0,T]}

\newcommand{\al}{\alpha}
\newcommand{\gam}{\gamma}
\newcommand{\del}{\delta}
\newcommand{\ep}{\varepsilon}

\newcommand{\sig}{\sigma}
\newcommand{\om}{\omega}

\newcommand{\Del}{\Delta}

\newcommand{\Om}{\Omega}

\newcommand{\ol}{\overline}

\newcommand{\pl}{\partial}
\newcommand{\supp}{{\rm supp}\,}

\newcommand{\bone}{\mathbf{1}}

\newcommand{\EP}{{\rm (E}_P{\rm)}}
\newcommand{\EQ}{{\rm (E}_Q{\rm)}}
\newcommand{\D}{{\rm (D}_\varepsilon{\rm)}}

\def\limssup{\mathop{\rm limsup\!^*}}
\def\limiinf{\mathop{\rm liminf_*}}

\begin{document}
\title[Homogenization of weakly coupled systems 
of HJ equations]
{Homogenization of weakly coupled systems of
Hamilton--Jacobi equations with fast switching rates}

\author[H. MITAKE]
{Hiroyoshi MITAKE}
\author[H. V. TRAN]
{Hung V. Tran}

\dedicatory{Dedicated to Professor H. Ishii on the occasion of his 65th birthday}

\address[H. Mitake]
{
Department of Applied Mathematics, 
Faculty of Science, Fukuoka University, 
Fukuoka 814-0180, Japan}
\email{mitake@math.sci.fukuoka-u.ac.jp}

\address[H. V. Tran]
{Department of Mathematics, 
University of California, Berkeley, CA 94720, USA}
\email{hung@math.uchicago.edu}


\thanks{This work was partially done while the first author 
visited Mathematics Department, University of California, Berkeley. }

\date{\today}

\begin{abstract}
We consider homogenization for weakly coupled systems 
of Hamilton--Jacobi equations with fast switching rates. 
The fast switching rate terms force the solutions converge to 
the same limit, which is a solution of the effective equation.
We discover the appearance of the initial layers,
which appear naturally when we consider the systems with 
different initial data and analyze them rigorously.
In particular, we obtain {matched} 
asymptotic solutions of the systems and rate of convergence.
We also investigate properties of the effective Hamiltonian of
weakly coupled systems and 
show some examples which do not appear in the context of single equations. 
\end{abstract}

\maketitle

\tableofcontents


\section{Introduction}
In this paper 
we study the behavior,
as 
$\ep (>0)$ tends to $0$, of
the viscosity solutions 
$(u^\ep_1,u^\ep_2)$ of
the following weakly coupled 
systems of Hamilton--Jacobi equations 
\begin{numcases}
{(\textrm{C}_\ep) \hspace{1cm}}
(u^\ep_{1})_t + H_{1}(\dfrac{x}{\ep},Du^\ep_{1}) 
+\dfrac{c_1}{\ep}( u^\ep_{1}-u^\ep_{2}) = 0
& in $\Q$, \nonumber \\
(u^\ep_{2})_t + H_{2}(\dfrac{x}{\ep},Du^\ep_{2}) 
+\dfrac{c_2}{\ep}(u^\ep_{2}-u^\ep_{1}) = 0
& in $\Q$, \nonumber \\
u^\ep_{i}(x,0)=f_{i}(x)
& 
on $\R^{n}$ for $i=1,2$, 
\nonumber
\end{numcases}
where 
$T>0$, $c_1, c_2$ are given positive constants and 
the Hamiltonians $H_{i}(\xi,p):\R^n \times \R^n\to\R$ are 
given continuous functions for $i=1,2$, 
which are assumed throughout the paper to satisfy the following. 
\begin{itemize}
\item[{\rm(A1)}]
The functions $H_{i}$ are  uniformly coercive in the $\xi$-variable, 
i.e., 
$$
\lim_{r\to\infty}\inf\{H_{i}(\xi,p)\mid \xi\in\R^n, 
|p|\ge r \}=\infty.
$$
\item[{\rm(A2)}]
The functions $\xi\mapsto H_{i}(\xi,p)$ are $\T^n$-periodic, 
i.e., 
$H_{i}(\xi+z,p)= H_{i}(\xi,p)$ 
for any $\xi,p\in\R^n$, $z\in\Z^n$ and $i=1,2$. 

\end{itemize}
The functions $f_{i}$ 
are given continuously differentiable functions 
on  $\R^{n}$ 
with $\| Df_i\|_{L^\infty(\R^n)}$ are bounded 
for $i=1,2$, respectively. 
Here $u^\ep_{i}$ are the real-valued unknown functions on $\cQ$ and 
$(u^\ep_{i})_t:=\partial u^\ep_{i}/\partial t, 
Du^\ep_{i}:=(\partial u^\ep_{i}/\partial x_1,\ldots,\partial u^\ep_{i}/\partial x_n)$ 
for $i=1,2$, respectively. 
We are dealing only with viscosity solutions of Hamilton--Jacobi equations 
in this paper and thus the term ``viscosity" may be omitted henceforth.

\subsection{Background: Randomly Switching Cost Problems}

System $(\textrm{C}_\ep)$ arises  as the dynamic programming
for the optimal control of the system
whose states are governed by certain ODEs, subject to random changes in
the dynamics: the system randomly switches at a fast rate $1/\ep$
among the two states. 
{
See \cite{D, EF, FS, BS, E} for instance.
}
Also see \cite{LY, CLLN} for another switching cost problems. 
In order to explain the background more precisely, 
we assume in addition that the  Hamiltonians $H_i$ are convex in $p$ here.
We define the functions $u_i^{\ep}:\cQ\to\R$ by 
\begin{equation}\label{def:value}
u_{i}^{\ep}(x,t):=
\inf\Big\{\E_{i}\Big(\int_{0}^{t}L_{\nu^{\ep}(s)}(\frac{\eta(s)}{\ep},-\dot{\eta}(s))\,ds
+f_{\nu^{\ep}(t)}(\eta(t))\Big)\Big\}, 
\end{equation}
where $L_i:\R^{2n}\to\R\cup\{+\infty\}$ are 
{the Fenchel-Legendre transform} of $H_i$, i.e., 
$L_i(\xi, q):=\sup_{p\in\R^n}(p\cdot q-H_i(\xi, p))$ 
for all $(\xi,q)\in\R^{2n}$ and 
the infimum is taken over $\eta\in\AC([0,t],\R^n)$ such that 
$\eta(0)=x$. 
Here 
$\AC([0,t],\R^n)$ denotes the set of absolutely continuous functions with 
value in $\R^n$ and 
$\E_{i}$ denotes the expectation of 
a process with $\nu^{\ep}(0)=i$ where 
$\nu^\ep$  is a $\{1,2\}$-valued process, which 
is a continuous-time Markov chain, such that 
\begin{equation}\label{markov}
\bP\big(\nu^{\ep}(s+\Del s)=j\mid 
\nu^{\ep}(s)=i\big)=\frac{c_i}{\ep}\Del s+o(\Del s) \ 
\textrm{as} \ \Del s\to0 \ \textrm{for} \ i\not=j, 
\end{equation}
where $o:[0,\infty)\to [0,\infty)$ is a function satisfying $o(r)/r\to0$
as $r\to0$. 
Formula \eqref{def:value} is basically the optimal control formula for the solution
of (C$_\ep$), where the random switchings among the two states are governed by
\eqref{markov}.

We first give a formal proof that $(u_1^\ep,u_2^\ep)$ given by 
\eqref{def:value} is a solution of (C$_{\ep}$).  
The rigorous derivation will be proved in Appendix by using the 
dynamic programming principle.
We suppose that $u_i^\ep\in C^1(\cQ)$ here. 
Set $u^\ep(x,i,t):=u_{i}^\ep(x,t)$ and 
$Y(s):=(\eta(s),\nu^{\ep}(s))$ for $\eta\in\AC(\R^n)$ with 
$\eta(0)=x$ and 
let $\nu^{\ep}$ be a Markov chain given by 
\eqref{markov} with $\nu^{\ep}(0)=i$. 
By Ito's formula for a jump process we have 
\begin{align*}
&
\E_{i}\Big(u^{\ep}(Y(t),0)-u^{\ep}(Y(0),t)\Big)\\
=&\, 
\E_{i}\Big(
\int_{0}^{t}-u_t^{\ep}(Y(s),t-s)+Du^{\ep}(Y(s),t-s)\cdot\dot{\eta}(s)\,ds\\
&\, 
+\int_{0}^{t}\sum_{j=1}^{2}
\big(u^{\ep}(\eta(s),j,t-s)-u^{\ep}(\eta(s),\nu^\ep(s),t-s)\big)\cdot\frac{c_{\nu^\ep(s)}}{\ep}\,ds
\Big)\\
\ge&\,  
\E_{i}\Big(
\int_{0}^{t}
-u^{\ep}_t(Y(s),t-s)
-H_{\nu^{\ep}(s)}(\frac{\eta}{\ep},Du^{\ep})
-L_{\nu^{\ep}(s)}(\frac{\eta}{\ep},-\dot{\eta})\,ds\\
&\, 
+\int_{0}^{t}\sum_{j=1}^{2}
\big(u^{\ep}(\eta(s),j,t-s)-u^{\ep}(\eta(s),\nu^\ep(s),t-s)\big)
\cdot \frac{c_{\nu^{\ep}(s)}}{\ep}\,ds
\Big)\\
=&\,  
-\E_{i}\Big(
\int_{0}^{t}L_{\nu^{\ep}(s)}(\frac{\eta}{\ep},-\dot{\eta})\,ds
\Big). 
\end{align*}
Thus, 
\[
u^{\ep}(x,i,t)\le
\E_{i}\Big(
\int_{0}^{t}L_{\nu^{\ep}(s)}(\frac{\eta}{\ep},-\dot{\eta})\,ds
+u^{\ep}(Y(t),0)
\Big). 
\]

In the above inequality, the equality holds if 
$-\dot{\eta}(s)\in D_{p}^{-}H_{\nu^\ep(s)}(\eta(s)/\ep, Du^\ep(Y(s),t-s))$, 
where $D_{p}^{-}H_i$ denotes the subdifferential of $H_i$ with respect to 
the $p$-variable.

\subsection{Main Results}
There have been extensively many important results on the study of homogenization
of Hamilton--Jacobi equations.
The first general result is due to Lions, Papanicolaou, and Varadhan \cite{LPV}
who studied the cell problems together with 
the effective Hamiltonian and established homogenization results under
quite general assumptions on Hamiltonians in the periodic setting.
The next major contributions to the subject are due to Evans
\cite{E,E2} who 
introduced the perturbed test function method in the framework of 
viscosity solutions. 
The method then has been adapted to study so many different homogenization problems that we cannot provide a complete list of references. 
We here only refer to the papers related to our work. 
Concordel \cite{C1,C2} 
achieved some first general results on the
properties of the effective Hamiltonian concerning flat parts
and non-flat parts.
Afterwards Capuzzo-Dolceta and Ishii \cite{CDI} combined the 
perturbed test functions with doubling variables methods to
obtain the first results on the rate of convergence of $u^\ep$ to $u$.
We refer to \cite{CCM, Tr1} for some recent progress.

There have been some interesting results \cite{S,CLL, CM}
on the study of homogenization for weakly coupled systems
of Hamilton--Jacobi equations in the periodic settings or 
in the almost periodic settings. 
We also refer to \cite{AS2} 
for a related result in the random setting. 
We refer the readers to \cite{EL, IK} for the complete theory of viscosity solutions
for weakly coupled systems of Hamilton--Jacobi and Hamilton--Jacobi--Bellman equations.
 Since 
the maximum principle and comparison principle still hold, 
homogenization results can be obtained by using the perturbed
test function method quite straightforwardly with some modifications.
Let us call attention also to the new interesting direction on
the large time behavior of weakly coupled systems of Hamilton--Jacobi equations,
which is related to homogenization through the cell problems.
The authors \cite{MT1}, and Camilli, Ley, Loreti and Nguyen \cite{CLLN} obtained
large time behavior results for some special cases 
but general cases still remain open.

Let us also refer to one of the main research directions in the study of homogenization,
stochastic homogenization of Hamilton--Jacobi equations,
which were first obtained
by Souganidis \cite{Sou1}, and Rezakhanlou and Tarver \cite{RT1}
independently.
See \cite{LS1, KRV, Sch1, LS3, AS1} for more
recent progress on the subject.

First we heuristically derive  
the behavior of solutions of (C$_{\ep}$) as $\ep$ tends to $0$. 
For simplicity, from now on, we \textit{always} assume that 
$c_1=c_2=1$. 
We consider the formal asymptotic expansions of solutions
$(u_1^{\ep}, u_2^{\ep})$ of (C$_{\ep}$) of the form 
\[
u_{i}^{\ep}(x,t):=
u_{i}(x,t)+\ep v_{i}(\frac{x}{\ep})+O(\ep^2). 
\]
Set $\xi:=x/\ep$. 
Plugging this into (C$_{\ep}$) and performing formal 
calculations, we achieve
\[
(u_{i})_{t}+\ldots 
+H_{i}(\xi, D_x u_i+D_\xi v_i+\cdots)
+\frac{1}{\ep}(u_{i}-u_{j})
+(v_{i}-v_{j})+\cdots
=0, 
\]
where we take $i, j\in\{1,2\}$ such that 
$\{i,j\}=\{1,2\}$. 
The above expansion implies that
$u_{1}=u_{2}=:u$. 
Furthermore, if we let $P=Du(x,t)$ then
$(v_1,v_2)$ is a $\T^n$-periodic solution of the following cell problem
\begin{numcases}
{({\rm E}_P) \hspace{1cm}}
H_{1}(\xi,P+Dv_{1}(\xi,P))+v_{1}(\xi,P)-v_{2}(\xi,P)= \ol{H}(P) 
& in  $\R^{n}$,\ \nonumber\\
H_{2}(\xi,P+Dv_{2}(\xi,P))+v_{2}(\xi,P)-v_{1}(\xi,P)= \ol{H}(P)
& in  $\R^{n}$,  \nonumber
\end{numcases}
where $\ol{H}(P)$ is a unknown constant. 
Because of the $\T^n$-periodicity of the Hamiltonians $H_i$,
we can also
consider the above cell problem on the torus $\T^n$,
which is equivalent to consider it on $\R^n$
with $\T^n$-periodic solutions.
By an argument similar to the classical one in \cite{LPV}, we have 
\begin{prop}[Cell Problems]\label{prop:cell}
For any $P\in\R^n$, there exists a unique constant $\ol{H}(P)$ 
such that
$({\rm E}_P)$ 
admits a $\T^n$-periodic solution
$(v_{1}(\cdot,P),v_{2}(\cdot,P))\in C(\R^n)^{2}$.  
We call $\ol{H}$ the effective Hamiltonian associated with $(H_1,H_2)$.
\end{prop}
See also \cite{CGT2, MT1, CLLN} for more details about the cell problems for weakly coupled systems.

Our main goal in this paper is threefold.
First of all, we want to demonstrate that $u_i^\ep$ converge locally uniformly to the same limit $u$
in $\R^n \times (0,T)$ for $i=1,2$ and $u$ solves
$$
u_t+\ol{H}(Du)=0
\quad \mbox{in } \Q.
$$
This part is a rather standard part in the study of homogenization of Hamilton--Jacobi equations
by using the perturbed test function method introduced by Evans \cite{E}
with some modifications. The only hard part comes from the fact that 
we do not have uniform bounds on the gradients of $u_i^\ep$ here because of
the fast switching terms. We overcome this difficulty by introducing the barrier functions 
(see Lemma \ref{lem:barrier}) and using the half-relaxed limits (see the proof of 
Theorem \ref{thm:main1}). The barrier functions furthermore give us 
the correct initial data for the limit $u$. 
Let $(u_{1}^{\ep}, u_{2}^{\ep})$ be the solution of (C$_{\ep}$) henceforth.

\begin{thm}[Homogenization Result] \label{thm:main1}
The functions $u^\ep_i$ converge locally uniformly in $\Q$ to the same limit 
$u\in C^{0,1}(\cQ)$ as $\ep \to 0$ for $i=1,2$ and $u$ solves
\begin{equation}\label{HJ.limit}
\begin{cases}
u_t+\ol{H}(Du)=0 \qquad \qquad\qquad &\textrm{in } \Q \\
u(x,0) = \ol{f}(x):=\dfrac{f_{1}(x)+f_{2}(x)}{2} &\textrm{on } \R^n.
\end{cases}
\end{equation}
\end{thm}
Formula \eqref{def:value} of solutions of (C$_{\ep}$) actually gives 
us an intuitive explanation about the effective initial datum $\ol{f}$.
As we send $\ep$ to $0$,
the switching rate becomes very fast and 
processes have to jump randomly {very quickly} 
between the two states with equal probability as given by \eqref{markov}.  
(Note that we assume $c_1=c_2=1$ now.) 
Therefore, it is relatively clear that $\ol{f}$ is the average of the 
given initial data $f_i$ for $i=1,2$. 
In general $\ol{f}$ depends on $c_1$ and $c_2$.

The second main part of this paper is the study of the initial layers appearing naturally
in the problem as the initial data of $u_i^\ep$ and $u$ are different in general.
We first study the initial layers in a heuristic mode by finding
inner and outer solutions, and using the matching asymptotic expansion method
to identify matched solutions (see Section 3.1).
We then combine the techniques of the matching asymptotic expansion method and
of Capuzzo-Dolceta and Ishii \cite{CDI} to obtain rigorously the rate of convergence 
result. 
\begin{thm}[{Rate of Convergence to Matched Solutions}]\label{thm:main2}
For each $T>0$, there exists $C:=C(T)>0$ such that
\[
\|u^\ep_{i} - m_{i}^{\ep}\|_{L^\infty(\R^n \times [0,T])}
 \le C \ep^{1/3} \ 
 \text{for} \ i=1,2, 
\]
where $u$ is the solution of \eqref{HJ.limit} and 
\begin{equation}\label{matched-func}
m_{i}^{\ep}(x,t):=u(x,t)+\dfrac{(f_{i}-f_{j})(x)}{2}e^{-\frac{2t}{\ep}}
\end{equation}
with $j\in\{1,2\}$ such that $\{i,j\}=\{1,2\}$. 
\end{thm}

Finally, we study various properties of the effective Hamiltonian $\ol{H}$.
It is always extremely hard to understand properties of the
effective Hamiltonians even for single equations.
Lions, Papanicolaou and Varadhan \cite{LPV} studied some preliminary
properties of the effective Hamiltonians and pointed out a $1$-dimensional example
that $\ol{H}$ can be computed explicitly.
After that, Concordel \cite{C1,C2} discovered some very interesting results related
flat parts and non-flat parts of $\ol{H}$ for more general cases. 
Evans and Gomes \cite{EG} found some further
properties on the strict convexity of $\ol{H}$ by using the weak KAM theory.

The properties of $\ol{H}$ for weakly coupled systems of Hamilton--Jacobi equations in this paper
are even more complicated. 
In case $H_1=H_2$, the effective Hamiltonian for the weakly
coupled systems and the single equations are obviously same. Therefore, we can view the cases of single
equations as special cases of the weakly coupled systems. However, in general, we cannot expect
the effective Hamiltonians for weakly coupled systems to have similar properties like single equations' cases.

The first few results on flat parts and non-flat parts of $\ol{H}$
 are {generalizations} to the ones discovered by Concordel \cite{C1, C2},
and are proved by using different techniques, 
namely the min-max formulas {which are derived in Section \ref{subsec:rep}} 
and the constructions of appropriate subsolutions.
On the other hand, we investigate other cases which show that 
the properties of the effective Hamiltonians for weakly coupled systems
are widely different from those of the effective Hamiltonians for single equations. 
Theorems \ref{thm:main6}, \ref{thm:main7}, \ref{thm:main7new} , \ref{thm:main8}, 
which are some of our main results, 
describe some rather new results which do not appear in the context 
of single Hamilton-Jacobi equations. 
Since the theorems are technical, we refer the readers 
to Section \ref{subsec:flat} for details. 
\smallskip

We are grateful to L. C. Evans for his suggestion which leads us to this project. 
We thank G. Barles,  D. Gomes, 
H. Ishii, T. Mikami, and F. Rezakhanlou for their fruitful discussions.  
We also thank S. Armstrong and P. E. Souganidis for letting us know about
the coming result on stochastic homogenization of 
weakly coupled systems of Hamilton--Jacobi equations of B. Fehrman \cite{Fe}.
Fehrman \cite{Fe} independently obtained interesting homogenization results 
of monotone systems of viscous Hamilton--Jacobi equations,
which are similar to ours,
with convex Hamiltonians in the stationary, ergodic setting
by using the ideas of Armstrong and Souganidis \cite{AS1, AS2}.
His work includes as well generalizations to other related systems.
\smallskip

The paper is organized as follows. 
In Section 2 we prove the homogenization result, Theorem \ref{thm:main1}. 
Section 3 devotes to the study of initial layers and rate of convergence.
We derive inner solutions, outer solutions, and matched asymptotic solutions in 
a heuristic mode and then prove Theorem \ref{thm:main2}.
The properties of the effective Hamiltonian are studied in Section 4.
We obtain its elementary properties in Section 4.1, the representation formulas
in Section 4.2, and flat parts, non-flat parts near the origin in Section 4.3. 
In Section 5 we prove generalization results for systems of $m$ equations for $m \ge 2$.
We then prove also the homogenization result for Dirichlet problems and describe
the differences of the effective data between Cauchy problems and Dirichlet problems 
in Section \ref{sec:dirichlet}.
Some lemmata concerning verifications of optimal control formulas for 
the Cauchy and Dirichlet problems are recorded in Appendix.

\medskip
\noindent
{\bf Notations.} 
{
For $k\in\N$ and $A\subset\R^n$, 
we denote by 
$C(A)$, $C^{0,1}(A)$ and $C^k(A)$ 
the space of real-valued 
continuous, 
Lipschitz continuous  
and $k$-th continuous differentiable functions  
on $A$, respectively. 
We denote $L^{\infty}(A)$ by the set of 
bounded measurable functions and 
$\|\cdot\|_{L^{\infty}(A)}$ denotes the superemum norm. 
Let $\T^n$ denote the $n$-dimensional torus 
 and  we identify $\T^n$ with $[0,1]^n$.
Define $\Pi:\R^n\to \T^n$ as the canonical projection.
By abuse of notations,
we denote
the periodic extensions of any set $B\subset \T^n$ and any function $f\in C(\T^n)$ 
to the whole space $\R^n$
by $B$, and $f$ themselves respectively.
For $a,b\in\R$, we write $a\land b=\min\{a,b\}$ and 
$a\vee b=\max\{a,b\}$.  
We call a function $m:[0,\infty)\to[0,\infty)$ a $modulus$ if it is continuous, 
nondecreasing on $[0,\infty)$ 
and $m(0)=0$. 
}


\section{Homogenization Results} \label{homo}

\begin{lem}[Barrier Functions]\label{lem:barrier}
We define the functions $\varphi_{i}^{\pm}:\cQ\to\R$ by 
\begin{equation}\label{func:barrier}
\begin{cases}
\varphi_1^{\pm}(x,t)=\dfrac{f_{1}(x)+f_{2}(x)}{2} 
+\dfrac{f_{1}(x)-f_{2}(x)}{2} e^{-\frac{2t}{\ep}}\pm Ct\\
\varphi_2^{\pm}(x,t)=\dfrac{f_{1}(x)+f_{2}(x)}{2} 
+\dfrac{f_{2}(x)-f_{1}(x)}{2} e^{-\frac{2t}{\ep}}\pm Ct. 
\end{cases}
\end{equation}
If we choose
$C\ge \max_{i=1,2}\max_{(\xi, p)\in \R^n \times B(0,r)} |H_{i}(\xi,p)|$, 
where $r=\|Df_1\|_{L^\infty(\R^n)}+\|Df_2\|_{L^\infty(\R^n)}$,
then $(\varphi_1^{-}, \varphi_2^{-})$ and $(\varphi_1^{+}, \varphi_2^{+})$ are, 
respectively, a subsolution and a supersolution of $({\rm C}_\ep)$,
and 
\[
(\varphi_1^{-}, \varphi_2^{-})(\cdot, 0)=(\varphi_1^{+}, \varphi_2^{+})(\cdot,0)=(f_{1},f_{2}) 
\ \text{on} \ \R^n.
\]
In particular, $\varphi_i^{-} \le u^\ep_i \le \varphi_i^{+}$ on $\cQ$ for $i=1,2$.
\end{lem}

\begin{proof}
We calculate that 
\begin{align*}
&
(\varphi_{1}^{-})_{t}
+H_{1}(\frac{x}{\ep}, D\varphi_{1}^{-})
+\frac{1}{\ep}(\varphi_{1}^{-}-\varphi_{2}^{-})\\
=&\, 
-\dfrac{f_{1}(x)-f_{2}(x)}{\ep} e^{-\frac{2t}{\ep}}-C
+H_{1}(\frac{x}{\ep}, D\varphi_{1}^{-})
+\dfrac{f_{1}(x)-f_{2}(x)}{\ep} e^{-\frac{2t}{\ep}}\\
=&\, 
-C
+H_{1}(\frac{x}{\ep}, D\varphi_{1}^{-}) \le 0
\end{align*}
for $C>0$ large enough as chosen above.
Similar calculations give us that 
$(\varphi_1^{-}, \varphi_2^{-})$ and $(\varphi_1^{+}, \varphi_2^{+})$ are, 
respectively, a subsolution and a supersolution of $(\textrm{C}_\ep)$. 
By the comparison principle for (C$_{\ep}$) (see \cite{EL,IK}) 
we get 
$\varphi_i^{-} \le u^\ep_i \le \varphi_i^{+}$ on $\cQ$ for $i=1,2$. 
\end{proof}

\begin{proof}[Proof of Theorem {\rm \ref{thm:main1}}]
By Lemma \ref{lem:barrier} we can take the following half-relaxed limits
$$
\begin{cases}
W(x,t):=\limssup_{\ep \to 0} \sup_{i=1,2}[u^\ep_i](x,t)\\
w(x,t):=\limiinf_{\ep \to 0} \inf_{i=1,2} [u^\ep_i] (x,t).
\end{cases}
$$
We now show that $W$ and $w$ are, respectively, 
a subsolution and a supersolution of \eqref{HJ.limit}
in $\Q$ by employing the perturbed test function method.
 
Since we can easily check $W(\cdot,0)=w(\cdot,0)=\ol{f}$ on $\R^n$ 
due to Lemma \ref{lem:barrier}, 
it is enough to prove that $W$ and $w$ are a subsolution and a supersolution, respectively, 
of the equation in \eqref{HJ.limit}. 
We only prove that $W$ is a subsolution since by symmetry we can prove that 
$w$ is a supersolution. 
We take a test function $\phi\in C^{1}(\Q)$ such that 
$W-\phi$ has a strict maximum at $(x_0,t_0) \in \Q$.
 Let $P:=D\phi(x_0,t_0)$. 
 Choose a sequence $\ep_m \to 0$ such that
 $$
 W(x_0,t_0)=\limssup_{m \to \infty}\max_{i=1,2} u^{\ep_m}_i(x_0,t_0).
 $$
We define the perturbed test functions $\psi^{\ep,\al}_i$ for $i=1,2$ 
and $\al>0$ by 
$$
 \psi^{\ep,\al}_i(x,y,t):=\phi(x,t)+\ep v_i(\dfrac{y}{\ep})+\dfrac{|x-y|^2}{2\al^2}, 
$$
where $(v_1,v_2)$ is a solution of $\EP$. 
By the usual argument in the theory of viscosity solutions,  
for every $m\in \N,\ \al>0$, there exist $i_{m,\al} \in \{1,2\}$ 
 and $(x_{m,\al},y_{m,\al},t_{m,\al}) \in \R^n \times \Q$
 such that 
 \begin{equation}\label{per.max}
 \max_{i=1,2} \max_{\R^n \times \cQ} 
 [u^{\ep_m}_i(x,t) - \psi^{\ep_m,\al}_i(x,y,t)]
 = u^{\ep_m}_{i_{m,\al}} (x_{m,\al},t_{m,\al}) - \psi^{\ep_m,\al}_{i_{m,\al}}(x_{m,\al},y_{m,\al},t_{m,\al})
 \end{equation}
 and up to passing some subsequences
\begin{align*}
&(x_{m,\al}, y_{m,\al} ,t_{m,\al}) \to (x_m, x_m, t_m)\ \text{as } \al \to 0,\\
&i_{m,\al} \to i_m \in \{1,2\} \ \text{as } \al \to 0,\\
&(x_m,t_m) \to (x_0,t_0)\ \text{as } m \to \infty,\\
&\lim_{m\to \infty} \lim_{\al \to 0} u^{\ep_m}_{i_{m,\al}}(x_{m,\al},t_{m,\al}) = W(x_0,t_0).
\end{align*}
 Choose $j_{m,\al}, j_m \in \{1,2\}$ such that 
 $\{i_{m,\al},j_{m,\al}\}=\{i_m, j_m\}=\{1,2\}$.
 By the definition of viscosity solutions, we have
 \begin{equation} \label{test1}
 \phi_t(x_{m,\al},t_{m,\al})+
 H_{i_{m,\al}} (\dfrac{x_{m,\al}}{\ep_m}, D\phi(x_{m,\al},t_{m,\al}) 
 +\dfrac{x_{m,\al}-y_{m,\al}}{\al^2})
 +\dfrac{1}{\ep_m} (u^{\ep_m}_{i_{m,\al}}-u^{\ep_m}_{j_{m,\al}})(x_{m,\al},t_{m,\al}) \le 0.
 \end{equation}
Since $(v_1,v_2)$ is a supersolution of $\EP$, we have
 \begin{equation} \label{test2}
 H_{i_{m,\al}}(\dfrac{y_{m,\al}}{\ep_m},P+\dfrac{x_{m,\al}-y_{m,\al}}{\al^2})
 +(v_{i_{m,\al}} -v_{j_{m,\al}})(\dfrac{y_{m,\al}}{\ep_m}) \ge \ol{H}(P).
 \end{equation}
 Let $\al \to 0$ in \eqref{test1} and \eqref{test2} to derive
 \begin{equation}\label{test3}
 \phi_t(x_{m},t_{m})+
 H_{i_{m}} (\dfrac{x_{m}}{\ep_m}, D\phi(x_{m},t_{m}) 
 +Q_m)
 +\dfrac{1}{\ep_m} (u^{\ep_m}_{i_{m}}-u^{\ep_m}_{j_{m}})(x_{m},t_{m}) \le 0
 \end{equation}
 and
  \begin{equation} \label{test4}
 H_{i_{m}}(\dfrac{x_{m}}{\ep_m},P+Q_m)
 +(v_{i_{m}} -v_{j_{m}})(\dfrac{x_{m}}{\ep_m}) \ge \ol{H}(P), 
 \end{equation}
 where $Q_m:= \lim_{\al \to 0} (x_{m,\al}-y_{m,\al})/\al^2$. 
 Noting that the correctors $v_i$ are Lipschitz continuous due to the coercivity of 
 $H_i$, we see that $|Q_m|\le C$ for $C>0$ which is independent of $m$. 
 Combine \eqref{test3} with \eqref{test4} to get
 \begin{align*}
 \phi_t (x_m,t_m) +\ol{H}(P) &\le
 H_{i_m}(\dfrac{x_m}{\ep_m},P+Q_m)-
   H_{i_m} (\dfrac{x_m}{\ep_m}, D\phi(x_m,t_m) + Q_m)\\
   &\quad+\dfrac{1}{\ep_m} 
   [u^{\ep_m}_{j_m} (x_m,t_m) - (\phi(x_m,t_m)+\ep_m v_{j_m}(\dfrac{x_m}{\ep_m}))]\\
   &\quad - \dfrac{1}{\ep_m}
    [u^{\ep_m}_{i_m} (x_m,t_m) - (\phi(x_m,t_m)+\ep_m v_{i_m}(\dfrac{x_m}{\ep_m}))]\\ 
    &\le \sig(| P - D\phi(x_m,t_m)|) 
 \end{align*}
for some modulus $\sig$.
Letting $m\to \infty$, we get the result.

We finally prove that $u$ is Lipschitz continuous. 
We can easily see that $\ol{f}\pm Mt$ are a supersolution and a subsolution 
of \eqref{HJ.limit}, respectively, for $M>0$ large enough. 
By the comparison principle for \eqref{HJ.limit} we have 
$|u(x,t)-\ol{f}(x)|\le Mt$ for all $(x,t)\in\cQ$. 
Moreover, the comparison principle for \eqref{HJ.limit} also yields that
\begin{align*}
&\sup_{x\in \R^n}|u(x,t+s)-u(x,t)|\le
\sup_{x\in \R^n} |u(x,s)-\ol{f}(x)|\le Ms \ \text{for all} \  t,s\ge 0, \ \text{and} \\
& \sup_{x \in \R^n} |u(x+z,t)-u(x,t)| \le
\sup_{x \in \R^n} |\ol{f}(x+z)-\ol{f}(x)| \le r|z| \ \text{for all} \ z \in \R^n, \ t \ge 0.
\end{align*}
The proof is complete.
\end{proof}


\section{Initial layers and Rate of convergence} \label{layer}

\subsection{Inner solutions, Outer solutions, and Matched solutions}
We first derive inner solutions, outer solutions and perform
 the matching asymptotic expansion method to find matched solutions
in a heuristic mode.

As we already obtained in Section \ref{homo}, 
outer solutions are same as the limit $u$ give in Theorem \ref{thm:main1}.
Now we need to find a right scaling for inner solutions.
We let
$$
w_i^\ep(x,t)=u_i^\ep(x,\ep t) \ \text{for } i=1,2,
$$
and plug into $\rm (C_\ep)$ to obtain
\begin{numcases}
{(\textrm{I}_\ep) \hspace{1cm}}
(w^\ep_{1})_t + \ep H_{1}(\dfrac{x}{\ep},Dw^\ep_{1}) 
+( w^\ep_{1}-w^\ep_{2}) = 0
& in $\R^n\times (0,T/\ep)$, \nonumber \\
(w^\ep_{2})_t +\ep H_{2}(\dfrac{x}{\ep},Dw^\ep_{2}) 
+(w^\ep_{2}-w^\ep_{1}) = 0
& in $\R^n \times (0,T/\ep)$, \nonumber \\
w^\ep_{i}(x,0)=f_{i}(x)
& 
on $\R^{n}$ for $i=1,2$.
\nonumber
\end{numcases}
We next assume that $w_i^\ep$ have the asymptotic expansions of the form
$$
w_i^\ep(x,t)=w_i(x,t)+\ep w_{i1}(x,t)+ \ep^2 w_{i2}(x,t) \cdots,
\ \text{for } i=1,2.
$$
It is then relatively straightforward to see that $(w_1,w_2)$ solves
\begin{numcases}
{(\textrm{I}) \hspace{1cm}}
(w_{1})_t 
+( w_{1}-w_{2}) = 0
& in $\R^n \times (0,\infty)$, \nonumber \\
(w_{2})_t 
+(w_{2}-w_{1}) = 0
& in $\R^n \times (0,\infty)$, \nonumber \\
w_{i}(x,0)=f_{i}(x)
& 
on $\R^{n}$ for $i=1,2$.
\nonumber
\end{numcases}
Thus, we can compute the explicit formula for the inner solutions 
\begin{align*}
&(w_1(x,t),w_2(x,t))\\
=&\, 
\left (\dfrac{f_1(x)+f_2(x)}{2} +\dfrac{f_1(x)-f_2(x)}{2} e^{-2t},\
\dfrac{f_1(x)+f_2(x)}{2} +\dfrac{f_2(x)-f_1(x)}{2} e^{-2t} \right).
\end{align*}
The final step is to obtain the matched solutions.
We have in this particular situation
$$
\lim_{t \to 0} u(x,t)=
\lim_{t \to \infty} w_i(x,t)=\dfrac{f_1(x)+f_2(x)}{2},
$$
which shows that the common part of the inner and outer solutions is 
$(f_1+f_2)(x)/2$.
Hence, the matched solutions are
\begin{align*}
&\left ( u(x,t)+w_1(x,\dfrac{t}{\ep}) - \dfrac{f_1(x)+f_2(x)}{2},
u(x,t)+w_2(x,\dfrac{t}{\ep}) - \dfrac{f_1(x)+f_2(x)}{2} \right )\\
=&\, 
\left ( u(x,t)+\dfrac{f_{1}(x)-f_{2}(x)}{2} e^{-\frac{2t}{\ep}},
u(x,t)+\dfrac{f_{2}(x)-f_{1}(x)}{2} e^{-\frac{2t}{\ep}} \right )\\
=&\, 
(m_1^\ep(x,t),m_2^\ep(x,t)), 
\end{align*}
where $m_{i}^{\ep}$ are the functions defined by \eqref{matched-func}.

As we can see, the matched solutions contain the layer parts which are essentially 
the same like the subsolutions and supersolutions that we build in Lemma \ref{lem:barrier}.
For any fixed $t>0$, we can see that $(m_1^\ep(x,t),m_2^\ep(x,t))$
converges to $(u(x,t),u(x,t))$ exponentially fast.
But for $t=O(\ep)$ then we do not have such convergence.
In particular, we have  $(m_1^\ep(x,\ep),m_2^\ep(x,\ep))$
converges to $\left (u(x,t)+(f_{1}-f_{2})(x)/(2e^{2}),
u(x,t)+(f_{2}-f_{1})(x)/(2e^{2}) \right)$.
On the other hand, the fact that $((m_1^\ep)_t, (m_2^\ep)_t)$
is not bounded also give us an intuition about the 
unboundedness of $((u_1^\ep)_t, (u_2^\ep)_t)$. 

It is therefore interesting if we can study the behavior of the difference
between the real solutions $(u_1^\ep,u_2^\ep)$
and the matched solutions $(m_1^\ep, m_2^\ep)$.

\subsection{Rate of convergence to matched solutions}\label{subsec:rate}

In this subsection, 
we assume further that
\begin{itemize}
\item[(A3)]
$H_i$ are (uniformly) Lipschitz in the $p$-variable for $i=1,2$, i.e.
there exists a constant $C_H>0$
such that 
\[
|H_{i}(\xi,p)-H_{i}(\xi,q)|\le C_H|p-q| \ 
\text{for all} \ \xi\in\T^n 
\ \text{and} \ p,q\in \R^n.
\] 
\end{itemize}

We now prove Theorem \ref{thm:main2} by 
splitting $\R^n \times [0,T]$ into two parts, which are
$\R^n \times [0,\ep|\log \ep|]$ and $\R^n \times [\ep|\log \ep|, T]$.
For the part of small time
$\R^n \times [0,\ep|\log \ep|]$, we use the
barrier functions in Lemma \ref{lem:barrier} and the effective equation to obtain
the results. The $L^\infty$-bounds of $|u_i^\ep-m_i^\ep|$ for $i=1,2$
on $\R^n \times [\ep|\log \ep|, T]$ can be obtained by using techniques
similar to those
of Capuzzo-Dolcetta and Ishii \cite{CDI}.

\begin{prop}[Initial Layer]\label{prop:layer}
There exists $C>0$ such that
$$
|(u^\ep_i-m_{i}^{\ep})(x,t)| \le C\ep|\log \ep|
\ 
\text{for all} \ (x,t)\in \R^n\times [0,\ep|\log \ep|] \ 
\text{and} \ i=1,2. 
$$
\end{prop}

\begin{proof}
We only prove the case $i=1$. 
By symmetry we can prove the case $i=2$. 
Pick a positive constant $C\ge \max_{i=1,2}\max_{(\xi, p)\in \R^n \times B(0,r)} |H_{i}(\xi,p)|$, 
where $r=\|Df_1\|_{L^\infty(\R^n)}+\|Df_2\|_{L^\infty(\R^n)}$ and 
note that $u$ is Lipschitz 
continuous with a Lipschitz constant 
$C_{u}:=(1/2)(\|Df_1\|_{\infty}+\|Df_2\|_{\infty})$. 
By Lemma \ref{lem:barrier} we have 
\begin{align*}
&\big|u^\ep_1(x,t)- (u(x,t)+\dfrac{f_{1}(x)-f_{2}(x)}{2} e^{-\frac{2t}{\ep}})\big|\\
\le\ & \big|u(x,t) - \dfrac{f_{1}(x)+f_{2}(x)}{2}\big| + Ct = |u(x,t)-u(x,0)| + Ct\\
\le \ & (C+C_{u})t \le (C+C_{u})\ep|\log \ep|
\end{align*}
for all $t \in [0,\ep|\log \ep|]$.
\end{proof}

\begin{prop} \label{prop:rate}
{Assume that {\rm(A3)} holds. }
For $T>0$ there exists $C=C(T)>0$ such that
$$
|u^\ep_i(x,t)-u(x,t)| \le C\ep^{1/3}
\quad \text{for } (x,t) \in \R^n \times [\ep|\log \ep|,T] \ \text{and} \ i=1,2.
$$
\end{prop}

\begin{lem} \label{app:lem1}
{Assume that {\rm(A3)} holds. }
For each $\del>0$ {and $P\in\R^n$,} 
there exists a unique solution $(v_1^\del, v_2^\del) \in C^{0,1}(\T^n)^2$ of 
\begin{numcases}
{({\rm E}^\del_P) \hspace{1cm}}
H_{1}(\xi,P+Dv^\del_{1}(\xi,P))+(1+\del)v^\del_{1}(\xi,P)-v^\del_{2}(\xi,P)= 0
& \textrm{in }  $\T^{n}$,\ \nonumber\\
H_{2}(\xi,P+Dv^\del_{2}(\xi,P))+(1+\del)v^\del_{2}(\xi,P)-v^\del_{1}(\xi,P)= 0
& \textrm{in } $\T^{n}$.  \nonumber
\end{numcases}
Moreover, 
\begin{itemize}
\item[(i)]  there exists a constant $C>0$ independent of 
$\del$ such that
$$
\del |v^\del_i(\xi,P) - v^\del_i(\xi,Q)| \le C|P-Q| 
\ \text{for } \xi\in \T^n, P,Q \in \R^n\ \text{and} \ i=1,2; 
$$
\item[(ii)] for each $R>0$, there exists a constant $C=C(R)>0$
independent of $\del$ such that
$$
|\del v_i^\del(\xi,P) + \ol{H}(P)| \le C(R) \del 
\ \text{for } \xi\in \T^n, P\in\ol{B}(0,R) \ \text{and} \ i=1,2.
$$
\end{itemize}
\end{lem}
\begin{proof}
By the classical result (see \cite{EL, IK}) we can easily see that 
there exists a unique solution $(v_1^{\del},v_2^{\del})$ 
of (E$_{P}^{\del}$) for any $P\in\R^n$. 
We see that 
$\left (v_1^\del(\cdot,P)\pm C_H |P-Q|/\del, 
v_2^\del(\cdot,P)\pm C_H |P-Q|/\del \right )$
are a supersolution and subsolution of (E$_Q^\del$), respectively, 
in view of (A3).
By the comparison principle for (E$_Q^\del$) we have 
$$
v_i^\del(\xi,P)-\dfrac{C_H |P-Q|}{\del}
\le v_i^\del(\xi,Q)
\le v_i^\del(\xi,P)+\dfrac{C_H |P-Q|}{\del}
\ \text{for} \ \xi \in \T^n, \ i=1,2,
$$
which completes (i).

Let $C_1(P)=\max_{i=1,2} \max_{\xi \in \T^n} |H_i(\xi,P)|$.
It is clear that 
{$(-C_1(P)/\del,-C_1(P)/\del)$ and $(C_1(P)/\del,C_1(P)/\del)$} 
are a subsolution and a supersolution of (E$_P^\del$), respectively.
Note that 
\[
|H_{i}(\xi,P)|\le |H_{i}(\xi,0)|+|H_{i}(\xi,0)-H_{i}(\xi,P)|
\le C(1+|P|) 
\]
for $C\ge \max_{i=1,2,\, \xi\in\T^n}|H_i(\xi,0)|\vee C_H$.
Therefore, 
by the comparison principle again we get 
\begin{equation}\label{CR1}
\del \| v_i^\del(\cdot,P) \|_{L^\infty(\T^n)} \le C_1(P) \le C(1+|P|). 
\end{equation}
Next, sum up the two equations of (E$_P^\del$) to get
$$
H_1(\xi,P+Dv_1^\del(\xi,P))+ H_2(\xi,P+Dv_2^\del(\xi,P)) \le 2C(1+|P|).
$$
Thus, for each $R>0$, there exists a constant $C=C(R)\ge0$ so that
\begin{equation}\label{CR2}
\|Dv_i^\del(\cdot,P)\|_{L^\infty(\T^n)} \le C(R) \ \text{for} \ |P| \le R 
\ \text{and} \ i=1,2.
\end{equation}
We look back at (E$_P^\del$) and take the inequalities \eqref{CR1}, \eqref{CR2}
into account to deduce that
\begin{equation}\label{CR3}
\|v_1^\del(\cdot,P)-v_2^\del(\cdot,P)\|_{L^\infty(\T^n)} \le C(R) \ \text{for} \ |P| \le R.
\end{equation}

Let 
$\mu^{+}:=\max_{i=1,2,\,\xi\in\T^n}\del v_{i}^{\del}(\xi,P)$ and 
$\mu^{-}:=\min_{i=1,2,\, \xi\in\T^n}\del v_{i}^{\del}(\xi,P)$. 
Then we have 
\begin{equation}\label{CR4}
\mu^{-}\le -\ol{H}(P)\le \mu^{+}.  
\end{equation}
Indeed, suppose that $\mu^{+}<-\ol{H}(P)$, then 
by the comparison principle we have 
$v_{i}^{\del}\ge w_{i}$ on $\T^n$ 
for any solution $(w_1,w_2,\ol{H}(P))$ of $\EP$. 
This is a contradiction, since for any $C_2\in\R$ 
$(w_1+C_2,w_2+C_2,\ol{H}(P))$ is a solution of $\EP$ too. 
Similarly we see that $\mu^{-}\le -\ol{H}(P)$. 
Combine \eqref{CR2}--\eqref{CR4} to get the desired conclusion of (ii). 
\end{proof}

{
We borrow some ideas from \cite{CDI} in the following proof.
}
\begin{proof}[Proof of Proposition {\rm \ref{prop:rate}}]
{
Let $(v_{1}^{\del}(\cdot, P), v_{2}^{\del}(\cdot, P))$ be 
the solution of (E$_{P}^{\del}$) for $P\in\R^n$. 
}
We consider the auxiliary functions
$$
\Phi_i(x,y,t,s)
:=
u^\ep_i(x,t)-u(y,s) - \ep v_i^\del(\dfrac{x}{\ep},\dfrac{x-y}{\ep^\beta})
-\dfrac{|x-y|^2+(t-s)^2}{2\ep^\beta} - K(t+s)
$$
for $i=1,2$,
where $\del=\ep^\theta$ and $\beta, \theta\in(0,1)$ and $K\ge0$ to be 
fixed later.

{
For simplicity of explanation we assume that 
$\Phi_i$ takes a global maximum on $\R^{2n}\times[\ep|\log\ep|,T]^{2}$ 
and let $(\hat{x},\hat{y},\hat{t},\hat{s})$ be a point such that 
\begin{equation} \label{layer:eqn1}
\max_{i=1,2} \max_{\R^n \times [\ep |\log \ep|, T]} \Phi_i(x,y,t,s)
= \Phi_1(\hat{x},\hat{y},\hat{t},\hat{s}).
\end{equation}
For a more rigorous proof we need to add the term $-\gam|x|^2$ to 
$\Phi_i$ for $\gam>0$. 
See the proof of Theorem 1.1 in \cite{CDI} for the detail. 
}
We first consider the case where $\hat{t}, \hat{s} >\ep |\log \ep|$. 
\smallskip

{\bf Claim.} 
If $0<\theta<1-\beta$, 
then there exists $M>0$ such that 
$(|\hat{x}-\hat{y}|+|\hat{t}-\hat{s}|)/\ep^\beta \le M$. 
\smallskip

We use $\Phi_1(\hat{x},\hat{y},\hat{t},\hat{s}) \ge \Phi_1(\hat{x},\hat{x},\hat{t},\hat{t})$, 
Lemma \ref{app:lem1} (i) and that $u$ is Lipschitz continuous to deduce that
\begin{align*}
\dfrac{|\hat{x}-\hat{y}|^2+|\hat{t}-\hat{s}|^2}{2\ep^\beta} &\le
|u(\hat{x},\hat{t})-u(\hat{y},\hat{s})| + 
\ep |v_1^\del(\dfrac{\hat{x}}{\ep},\dfrac{\hat{x}-\hat{y}}{\ep^\beta})
-v^\del_1(\dfrac{\hat{x}}{\ep},0)| + K|\hat{t}-\hat{s}|\\
\le&\, C_{u}(|\hat{x}-\hat{y}|+|\hat{t}-\hat{s}|) 
+C \ep \dfrac{1}{\ep^\theta} \dfrac{|\hat{x}-\hat{y}|}{\ep^\beta}
+K|\hat{t}-\hat{s}|\\
\le&\, 
C^{'}(|\hat{x}-\hat{y}|+|\hat{t}-\hat{s}|) 
\end{align*}
for some $C, C^{'}>0$, 
which implies the desired result.

\smallskip

We fix $(y,s)=(\hat{y},\hat{s})$ and notice that the function 
$$ 
(x,t) \mapsto u^\ep_1(x,t)- 
\ep v_1^\del(\dfrac{x}{\ep},\dfrac{x-\hat{y}}{\ep^\beta})
-\dfrac{|x-\hat{y}|^2+(t-\hat{s})^2}{2\ep^\beta} - Kt
$$ 
attains the maximum at $(\hat{x},\hat{t})$. 
For $\al>0$, we define the function $\psi$ by 
$$
\psi(x,\xi,z,t):=
 u^\ep_1(x,t)- 
\ep v_1^\del(\xi,\dfrac{z-\hat{y}}{\ep^\beta})
-\dfrac{|x-\hat{y}|^2+|t-\hat{s}|^2}{2\ep^\beta}
-\dfrac{|x-\ep \xi|^2+|x-z|^2}{2\al} - Kt. 
$$
Let $\psi$ attain the maximum at $(x_\al,\xi_\al, z_\al, t_\al)$ and 
then we may assume that 
$(x_\al,\xi_\al, z_\al, t_\al) \to (\hat{x},\hat{x}/\ep,\hat{x},\hat{t})$
as $\al \to 0$ up to passing a subsequence if necessary.
By the definition of viscosity solutions, we have
\begin{equation} \label{layer:eqn2}
K+\dfrac{t_\al-\hat{s}}{\ep^\beta}+
H_1(\dfrac{x_\al}{\ep},\dfrac{x_\al-\hat{y}}{\ep^\beta}+
\dfrac{x_\al-\ep \xi_\al}{\al} + \dfrac{x_\al-z_\al}{\al})
+\dfrac{1}{\ep}(u_1^\ep-u_2^\ep)(x_\al,t_\al) \le 0,
\end{equation}
and
\begin{equation} \label{layer:eqn3}
H_1(\xi_\al,\dfrac{z_\al-\hat{y}}{\ep^\beta}+\dfrac{x_\al -\ep \xi_\al}{\al}) 
+ (1+\del) v_1^\del(\xi_\al, \dfrac{z_\al -\hat{y}}{\ep^\beta})
 - v_2^\del(\xi_\al, \dfrac{z_\al -\hat{y}}{\ep^\beta}) \ge 0.
\end{equation}
Next, since $\psi(x_\al,\xi_\al,z_\al,t_\al) \ge \psi(x_\al,\xi_\al,x_\al,t_\al)$
we get 
$$
\dfrac{|x_\al-z_\al|^2}{2\al} \le
\ep (v_1^\del(\xi_\al,\dfrac{x_\al-\hat{y}}{\ep^\beta})-
v_1^\del(\xi_\al,\dfrac{z_\al-\hat{y}}{\ep^\beta}))
\le C \ep^{1-\theta-\beta} |x_\al -z_\al|
$$
by Lemma \ref{app:lem1} (i). 
Thus, $|x_\al-z_\al|/\al \le C\ep^{1-\theta-\beta}$.
Combine this with \eqref{layer:eqn2} and \eqref{layer:eqn3}, and 
send $\al \to 0$ to yield 
\begin{equation}\label{layer:eqn4}
K+\dfrac{\hat{t}-\hat{s}}{\ep^\beta}+
\ol{H}(\dfrac{\hat{x}-\hat{y}}{\ep^\beta})+
\dfrac{1}{\ep}(u_1^\ep-u_2^\ep)(\hat{x},\hat{t})
-v_1^\del(\dfrac{\hat{x}}{\ep},\dfrac{\hat{x}-\hat{y}}{\ep^\beta})+
v_2^\del(\dfrac{\hat{x}}{\ep},\dfrac{\hat{x}-\hat{y}}{\ep^\beta})-
C(\ep^\theta+\ep^{1-\theta-\beta}) \le 0.
\end{equation}
Similarly we fix $(x,t)=(\hat{x},\hat{t})$ and do a similar procedure 
to the above to obtain
\begin{equation}\label{layer:eqn5}
-K+\dfrac{\hat{t}-\hat{s}}{\ep^\beta}+\ol{H}(\dfrac{\hat{x}-\hat{y}}{\ep^\beta})+
C(\ep^\theta+\ep^{1-\theta-\beta}) \ge 0.
\end{equation}
Combining \eqref{layer:eqn4}, \eqref{layer:eqn5}, and \eqref{layer:eqn1}, we get
\begin{equation}\label{layer:eqn6}
2K \le C(\ep^\theta+\ep^{1-\theta-\beta}).
\end{equation} 
Now we choose $\theta=\beta=1/3$ and $K=K_1 \ep^{1/3}$ for $K_1$ large enough
to get the contradiction in \eqref{layer:eqn6}.
Hence either $\hat{t}=-\ep\log\ep$ or $\hat{s}=-\ep\log\ep$ holds.
The proof is complete immediately.
\end{proof}

{
Theorem \ref{thm:main2} is a straightforward result of 
Propositions \ref{prop:layer}, \ref{prop:rate}. 
}


\section{Properties of effective Hamiltonians}

\subsection{Elementary properties}\label{subsec:ele}
\begin{prop}\label{prop:elementary-effective} \ \\
{
{\rm (i)} \ 
{\rm (}Coercivity{\rm )} \ 
$\ol{H}(P)\to +\infty$ as $|P|\to\infty$. 
\\
{\rm (ii)} 
{\rm (}Convexity{\rm )} \ 
If $H_i$ are convex in the $p$-variable for $i=1,2$,  
then $\ol{H}$ is convex.
}
\end{prop}

\begin{proof}
(i) 
For each $\del>0$ and $P\in\R^n$,
let 
$(v_{1}^{\del}, v_{2}^{\del})$ be a solution of (E$_{P}^{\del}$) and 
without loss of generality, we may assume that
$v_{1}^{\del}(\xi_0,P)=\max_{i=1,2,\, \xi\in\T^n}v_{i}^{\del}(\xi,P)$
for some $\xi_0 \in \T^n$.
By the definition of viscosity solutions we have 
$H_1(\xi_0,P)\le H_{1}(\xi_0,P)+(v_{1}^{\del}-v_{2}^{\del})(\xi_0,P)\le 
-\del v_{1}^{\del}(\xi_0,P)$. 
We let $\del \to 0$ to derive that
$\ol{H}(P) \ge \min_{i=1,2,\, \xi \in\T^n} H_i(\xi,P)$.
Since $H_i$ are coercive for $i=1,2$, so is $\ol{H}$.

(ii) We argue by contradiction. 
Suppose that $\ol{H}$ is not convex and 
then there would exist $P, Q \in \R^n$ such that
\begin{equation}\label{eff1}
2\ep_0:=\ol{H}(\dfrac{P+Q}{2})-\dfrac{\ol{H}(P) + \ol{H}(Q)}{2}>0. 
\end{equation}
We define the functions $w_i\in C(\T^n)$ so that
$w_i(\xi):=(v_i(\xi,P)+v_i(\xi,Q))/2$ for $i=1,2$, 
where $(v_{1}(\cdot, P), v_{2}(\cdot, P))$ and 
$(v_{1}(\cdot, Q), v_{2}(\cdot, Q))$ are solutions of 
$\EP$ and $\EQ$, respectively. 
Due to the convexity of $H_i$ for $i=1,2$ we have 
$$
\begin{cases}
H_1(\xi,\dfrac{P+Q}{2}+Dw_1(\xi))
+w_1(\xi)-w_2(\xi) 
\le \dfrac{\ol{H}(P)+\ol{H}(Q)}{2},\\
H_2(\xi,\dfrac{P+Q}{2}+Dw_2(\xi))
+w_2(\xi)-w_1(\xi)
\le \dfrac{\ol{H}(P)+\ol{H}(Q)}{2}.
\end{cases}
$$

By \eqref{eff1},
there exists a small constant $\del>0$ such that
$$
\begin{cases}
H_1(\xi,\dfrac{P+Q}{2}+Dw_1(\xi))
+(1+\del)w_1(\xi)-w_2(\xi) 
\le \ol{H}(\dfrac{P+Q}{2})-\ep_0,\\
H_2(\xi,\dfrac{P+Q}{2}+Dw_2(\xi))
+(1+\del)w_2(\xi)-w_1(\xi)
\le  \ol{H}(\dfrac{P+Q}{2})-\ep_0, 
\end{cases}
$$
and
$$
\begin{cases}
H_1(\xi,\dfrac{P+Q}{2}+Dv_1(\xi,\dfrac{P+Q}{2}))
+(1+\del)v_1(\xi,\dfrac{P+Q}{2})-v_2(\xi,\dfrac{P+Q}{2})
\ge \ol{H}(\dfrac{P+Q}{2})-\ep_0,\\
H_2(\xi,\dfrac{P+Q}{2}+Dv_2(\xi,\dfrac{P+Q}{2}))
+(1+\del)v_2(\xi,\dfrac{P+Q}{2})-v_1(\xi,\dfrac{P+Q}{2})
\ge \ol{H}(\dfrac{P+Q}{2})-\ep_0.
\end{cases}
$$
The usual comparison principle implies that
\begin{equation}\label{eff2}
\dfrac{v_i(\xi,P)+v_i(\xi,Q)}{2} \le v_i(\xi,\dfrac{P+Q}{2}) 
\qquad \mbox{for } i=1,2.
\end{equation}
Notice that \eqref{eff2} is still correct even if we replace
$v_i(\xi,(P+Q)/2)$ by $v_i(\xi,(P+Q)/2)+C_1$
for $i=1,2$ and for any $C_1 \in \R$, which yields
the contradiction.
\end{proof}

The uniqueness of the effective Hamiltonian for $\EP$ and 
the cell problem for single Hamilton--Jacobi equations give 
the following proposition. 
\begin{prop}\label{prop:ele1}
If $H_1=H_2=K$, then
$$
\ol{H}(P)=\ol{K}(P) \
\text{for all } P \in \R^n,
$$
where $\ol{K}$ is the effective Hamiltonian corresponding to $K$.
\end{prop}

{
\begin{prop}\label{prop:ele2}
If $H_i$ are homogeneous with degree $1$ in the $p$-variable 
for $i=1,2$, 
then $\ol{H}$ is positive homogeneous with degree $1$.
\end{prop}
\begin{proof}
Let $(v_1,v_2,\ol{H}(P))$ be a solution of $\EP$ for any $P\in\R^n$. 
If $H_i$ is homogeneous with degree $1$ in the $p$-variable, 
then $(rv_1,rv_2,r\ol{H}(P))$ is a solution of (E$_{rP}$) for 
any $r>0$. 
Therefore by the uniqueness of the effective Hamiltonian 
we get the conclusion. 
\end{proof}
}

\begin{prop} \label{prop:max-H1H2}
We define the Hamiltonian $K$ as
$$
K(\xi,p):=\max \{ H_1(\xi,p), H_2(\xi,p) \}. 
$$
Let $\ol{K}$ be its corresponding effective Hamiltonian,then for all $P\in \R^n$,
$$
\ol{H}(P) \le \ol{K}(P).
$$
\end{prop}

\begin{proof}
For each $P \in \R^n$, there exists $\varphi(\cdot, P) \in C^{0,1}(\T^n)$
such that
$$
K(\xi,P+D\varphi(\xi,P))=\ol{K}(P).
$$
Thus $(\varphi(\cdot,P),\varphi(\cdot,P),\ol{K}(P))$ is a subsolution of $(\textrm{E}_P)$.
We hence get $\ol{K}(P) \ge \ol{H}(P)$ by Proposition \ref{prop:rep-1}.
\end{proof}

{
We give an example that we can calculate the effective Hamiltonian explicitly. 
}
\begin{example}
Let $n=1$ and $H_1(\xi,p)=|p|$, $H_2(\xi,p)=a(\xi)|p|$, where
$$
a(\xi):=\dfrac{1-(\frac{1}{8\pi^2}\cos(2\pi \xi)+\frac{1}{4\pi}\sin(2\pi \xi))}
{1+(\frac{1}{2}+\frac{1}{8\pi^2}) \cos(2\pi \xi)}
{>0}.
$$
By Proposition \ref{prop:ele2} we have 
$\ol{H}(P)=\ol{H}(1)P$ for $P \ge 0$.
Set 
$$
v_1(\xi,1):=\dfrac{1}{16\pi^3} \sin(2\pi \xi) -\dfrac{1}{8\pi^2} \cos (2\pi \xi), \
v_2(\xi,1):=( \dfrac{1}{4\pi}+\dfrac{1}{16\pi^3} ) \sin(2\pi \xi).
$$
Then we can confirm that $(v_1(\cdot,1), v_2(\cdot,1),1)$ 
is a solution of (E$_1$). 
Therefore $\ol{H}(1)=1$ and thus, 
$\ol{H}(P)=P$ for $P \ge 0$.

{
For any $P<0$ we have 
$\ol{H}(P)=\ol{H}(-1)\cdot(-P)$.
Set 
$$
v_1(\xi,-1):=-\big(\dfrac{1}{16\pi^3} \sin(2\pi \xi)-\dfrac{1}{8\pi^2} \cos (2\pi \xi)\big), \
v_2(\xi,1):=-( \dfrac{1}{4\pi}+\dfrac{1}{16\pi^3} ) \sin(2\pi \xi).
$$
It is straightforward to check that $(v_1(\cdot,-1), v_2(\cdot,-1),1)$ 
is a solution of (E$_{-1}$). 
Therefore $\ol{H}(-1)=1$ and  
$\ol{H}(P)=-P$ for $P \le 0$.}
We get $\ol{H}(P)=|P|$. 
\end{example}


\subsection{Representation formulas for the effective Hamiltonian}\label{subsec:rep}
In this subsection we derive representation formulas for the effective Hamiltonian 
$\ol{H}(P)$. 
See \cite{CIPP, Gomes1} for the min-max formulas for the
effective Hamiltonian for single equations.

\begin{prop}[Representation formula 1]\label{prop:rep-1}
We have 
\begin{multline}
\ol{H}(P) = \inf \{c:\ \text{there exists} \ (\phi_1,\phi_2)\in C(\T^n)^2 
\ \text{so that}\\ 
\text{the triplet } (\phi_1,\phi_2,c) 
\ \text{is a subsolution of $\rm (E_P)$} \}. \label{rep:eff-1}
\end{multline}
\end{prop}

\begin{proof}
Fix $P\in\R^n$ and 
we denote by $c(P)$ the right-hand  side 
of \eqref{rep:eff-1}. 
By the definition of $c(P)$ 
we can easily see that $\ol{H}(P)\ge c(P)$. 
We prove the other way around. 
Assume by contradiction that
there exist a triplet $(\phi_1,\phi_2,c)\in C(\T^n)^2 \times \R$ which is 
a subsolution of $\EP$ and $c<\ol{H}(P)$. 
Let $(v_1, v_2, \ol{H}(P))$ be a solution of $\EP$ and  
take $C>0$ so that $\phi_i>v_i-C=:\ol{v}_{i}$ on $\T^n$. 
Then since $\ol{v}_i$ and $\phi_i$ are bounded on $\T^n$, 
for $\ep>0$ small enough, we have
$$
\begin{cases}
H_1(\xi,P+D\ol{v}_1)+ (1+\ep)\ol{v}_1 - \ol{v}_2 \ge 
H_1(\xi,P+D\phi_1) + (1+\ep) \phi_1 -\phi_2\\
H_2(\xi,P+D\ol{v}_2) + (1+\ep)\ol{v}_2 - \ol{v}_1
 \ge H_2(\xi,P+D\phi_2) + (1+\ep) \phi_2 - \phi_1.
\end{cases}
$$
By the comparison principle (see \cite{EL,IK}) 
we deduce $\ol{v}_i \ge \phi_i$ on $\T^n$  
which yields the contradiction.
\end{proof}

If we assume the convexity on $H_i(\xi,\cdot)$ for any $\xi\in\R^n$, 
by the classical result on the representation formula for the effective Hamiltonian 
for single Hamilton--Jacobi equations 
we can easily see that 
{
\begin{align}
\ol{H}(P)&=\inf_{\varphi \in C^{1}(\T^n)} 
\max_{\xi \in \T^n} [H_1(\xi,P+D\varphi(\xi)) + v_1(\xi,P)-v_2(\xi,P)]
\label{rep-implicit}\\
&=\inf_{\psi \in C^{1}(\T^n)} \max_{\xi \in \T^n} 
[H_2(\xi,P+D\psi(\xi))+v_2(\xi,P)-v_1(\xi,P)] 
\nonumber
\end{align}
}
for any solution $(v_1(\cdot,P),v_2(\cdot,P))$  of $\EP$, 
{which is in a sense an implicit formula.}  
For the weakly coupled system we have the following representation formula. 

\begin{prop}[Representation formula 2] \label{prop:rep-2}
If $H_i$ are convex in the $p$-variable for $i=1,2$, 
then
\begin{equation}\label{rep:eff-2}
\ol{H}(P)=\inf_{(\phi_1,\phi_2)\in C^{1}(\T^n)^2} 
\max_{i=1,2,\, \xi \in \T^n} [H_i(\xi,P+D\phi_{i}(\xi)) + \phi_i(\xi)-\phi_j(\xi)],  
\end{equation}
where we take $j\in\{1,2\}$ so that $\{i,j\}=\{1,2\}$. 
\end{prop}

\begin{lem}
{Assume that $H_i$ are convex in the $p$-variable.} 
Let $(v_1,v_2,\ol{H}(P))\in C(\T^n)^2$ be a subsolution of 
$\EP$. For $\del>0$,
set $v_{i\del}(x):=\rho_{\del}\ast v_{i}(x)$, 
where 
$\rho_\del(x):=\del^{-n}\rho(x/\del)$ with 
$\rho\in C^{\infty}(\R^n)$ be a standard mollification kernel, 
i.e., $\rho\ge0$, $\supp \rho\subset B(0,1)$, 
and $\int_{\R^n}\rho(x)\,dx=1$.
Then $(v_{1\del},v_{2\del},\ol{H}(P)+\om(\del))$ 
is a subsolution of $\EP$ for some modulus $\om$. 
\end{lem}

{
\begin{proof}
Note that in view of the coercivity of $H_i$,
$v_i$ are Lipschitz continuous and
$(v_1,v_2,\ol{H}(P))$ solves $\EP$
almost everywhere.  
Fix any $\xi\in \T^n$. 
We calculate that 
\begin{eqnarray*}
\ol{H}(P)&\ge& 
\rho_{\del}\ast 
\big(H_1(\cdot,Dv_1(\cdot))+(v_1-v_2)\big)(\xi)\\
{}&=& 
\int_{B(\xi,\del)}\rho_{\del}(\xi-\eta)
\big(H_1(\eta,Dv_1(\eta))+(v_1-v_2)(\eta)\big)\,d\eta\\
{}&\ge& 
\int_{B(\xi,\del)}\rho_{\del}(\xi-\eta)
\big (H_1(\xi,Dv_1(\eta))-\om(\del) \big )\,d\eta +(v_{1\del}-v_{2\del})(\xi)\\
{}&\ge& 
H_1(\xi,\rho_{\del}\ast Dv_1(\xi))+(v_{1\del}-v_{2\del})(\xi)-\om(\del)\\
{}&=&
H_{1}(\xi,Dv_{1\del}(\xi))+(v_{1\del}-v_{2\del})(\xi)-\om(\del),
\end{eqnarray*}
where the third inequality follows by using Jensen's inequality. 
Here $\om$ is a modulus of continuity.
\end{proof}
}

\begin{proof}[Proof of Proposition {\rm \ref{prop:rep-2}}]
{
Let $c(P)$ be the constant on the right-hand  side of \eqref{rep:eff-2}. 
Noting that for any $(\phi_{1},\phi_2)\in C^{1}(\T^n)^2$ 
\[
H_{i}(\xi,P+D\phi_{i}(\xi))+(\phi_{i}-\phi_{j})(\xi)
\le \max_{i=1,2,\, \xi\in\T^n}[H_{i}(\xi,P+D\phi_i(\xi))
+(\phi_{i}-\phi_{j})(\xi)]=:a_{\phi_1,\phi_2}\]
for every $\xi \in \T^n$.
By Proposition \ref{prop:rep-1} we see that 
$\ol{H}(P)\le a_{\phi_1,\phi_2}$ 
for all $(\phi_{1},\phi_2)\in C^{1}(\T^n)^2$. 
Therefore we get $\ol{H}(P)\le c(P)$.  
}

{
Conversely, 
we observe that by Proposition \ref{prop:cell}  
$(v_{1\del}(\cdot,P), v_{2\del}(\cdot,P), \ol{H}(P)+\om(\del))\in C^{1}(\T^n)^2\times\R$  
is a subsolution of $\EP$. 
Therefore, by the definition of $c(P)$ we see that 
 $c(P)\le \ol{H}(P)+\om(\del)$. 
Sending $\del\to0$ yields the conclusion.  
}
\end{proof}

{
If $H_{i}$ are convex in the $p$-variable, then there is a variational 
formula for solutions of the initial value problem and the cell problem 
as stated in Introduction. 
Therefore, naturally 
we have the following variational formula  
\begin{align*}
\ol{H}(P)=&\, 
-\lim_{\del\to0}
\inf_{\eta} 
\E_{i}\Bigl[
\int_{0}^{+\infty}e^{-\del s} 
\bigl(-P\cdot \dot{\eta}(s)+L_{\nu(s)}(\eta(s),-\dot{\eta}(s))\bigr)
\,ds
\Bigr]\\
=&\, 
-\lim_{t\to\infty}
\frac{1}{t} 
\inf_{\eta}
\E_{i}\Bigl[
\int_{0}^{t}
\bigl(-P\cdot\dot{\eta}(s)+L_{\nu(s)}(\eta(s),-\dot{\eta}(s))\bigr)
\,ds
\Bigr], 
\end{align*}
where the infimum is taken over $\eta\in\AC([0,+\infty),\R^n)$ such that 
$\eta(0)=x$ and 
$\E_{i}$ denotes the expectation of 
a process with $\nu(0)=i$ given by \eqref{markov}. 
}

\begin{rem} 
When we consider the nonconvex Hamilton--Jacobi equations, 
in general we cannot expect the formula \eqref{rep:eff-2}. 
Take the Hamiltonian 
\begin{equation}\label{example}
H_{i}(\xi,p):=(|p|^2-1)^2 \ \text{for} \ i=1,2 
\end{equation}
for instance. 
In this example if we calculate the right-hand  side of \eqref{rep:eff-2} 
with $P=0$, then it is $0$. 
But we can easily check that $\ol{H}(0)=1$, 
since in this case we can choose $v_1(\cdot,0)=v_2(\cdot,0)\equiv 0$ to be a solution of (E$_0$).

The following formula is a revised min-max formula 
for the effective Hamiltonian for nonconvex Hamilton--Jacobi equations. 
{
\begin{prop}
We have 
\begin{equation}\label{rep:eff-3}
\ol{H}(P)=
\inf_{(\phi_1,\phi_2)\in C^{0,1}(\T^n)^2}\max_{i=1,2,\,\xi \in\T^n} 
\sup_{p\in D^{+}\phi_i(\xi)}
[H_{i}(\xi,P+p)+(\phi_i-\phi_j)(\xi)], 
\end{equation}
where 
if $D^{+}\phi_i(\xi)=\emptyset$,
 then we set 
$\sup_{p\in D^{+}\phi_i(\xi)}[H_{i}(\xi,P+p)+(\phi_i-\phi_j)(\xi)]=-\infty$ by convention.  
\end{prop}
}
We notice that 
if $H_i$ are given by \eqref{example},
then the right-hand  side of \eqref{rep:eff-3} 
{with $P=0$} is $1$. 

\begin{proof}
The proof is already in the proof of Proposition \ref{prop:rep-2}. 
We just need to be careful for 
the definition of viscosity subsolutions. 
Indeed, let $c$ be the right-hand  side of \eqref{rep:eff-3} 
and noting that for any 
$(\phi_1,\phi_2)\in C^{0,1}(\T^n)^2$,
$\xi \in \T^n$, and $q \in D^{+} \phi_i(\xi)$,
\begin{align*}
&H_{i}(\xi,P+q)+(\phi_i-\phi_j)(\xi)\le
\max_{\xi \in\T^n,i=1,2}\sup_{p\in D^{+}\phi_i(\xi)}
[H_i(\xi,P+p)-(\phi_i-\phi_j)(\xi)]=:a_{\phi_1,\phi_2}.
\end{align*}
Thus, $\ol{H}(P)\le a_{\phi_1,\phi_2}$ 
for all $(\phi_1,\phi_2)\in C^{0,1}(\T^n)^2$ 
by Proposition \ref{prop:rep-1}. 
Therefore, $\ol{H}(P)\le c$.

Conversely, 
there exists a viscosity subsolution $(v_1(\cdot,P), v_2(\cdot,P), \ol{H}(P))\in C^{0,1}(\T^n)^2\times\R$  
of $\EP$.
By the definition of viscosity subsolutions we have 
\[
H_{i}(\xi,P+p)+(v_i-v_j)(\xi) \le \ol{H}(P) \ 
\textrm{for all} \ \xi \in\T^n \ 
\textrm{and} \ p\in D^{+}v_i(\xi).  
\]
Thus, 
\[
\max_{\xi\in\T^n,i=1,2}
\sup_{p\in D^{+}v_{i}(\xi)}
[H_i(\xi,P+p)+(v_i-v_j)(\xi)] \le \ol{H}(P),  
\]
which implies $c\le \ol{H}(P)$. 
\end{proof}

\end{rem}


\subsection{Flat parts and Non-flat parts near the origin}\label{subsec:flat}

{
In this subsection, we study the results concerning flat parts and non-flat parts
of the effective Hamiltonian $\ol{H}$ near the origin.
We first point out that there are some cases in which we can obtain similar results
to those of Concordel's results for single equations. We 
present different techniques to obtain these results
, namely the min-max formulas, and the construction
of subsolutions.
In this subsection, we {only deal with the Hamiltonians of the form} 
$H_i(\xi,p)=|p|^2 - V_i(\xi)$, where 
$V_i \in C(\T^n)$ for $i=1,2$ unless otherwise stated.
}

\begin{thm}\label{thm:main4}
Assume that  $V_i \ge 0$ in $\T^n$ and $\{V_i=0\}=:U_i \subset \T^n$ 
for  $i=1,2$.
We assume further that $U_1 \cap U_2 \ne \emptyset$
and there exist open sets $W_1, W_2$ in $\T^n$,
and a vector $q\in \R^n$ such that $\Pi(q+W_2) \Subset (0,1)^n$ and
\begin{equation}\label{flat.con}
U_1 \cup U_2 \subset W_1 \subset W_2 \
\text{and } {\rm dist}(W_1, \partial W_2),\  {\rm dist}(U_1 \cup U_2, \partial W_1)>0, 
\end{equation}
then there exists $\gam>0$ such that
$\ol{H}(P)=0$ for $|P| \le \gam$.
\end{thm}

\begin{center}
\unitlength 0.1in
\begin{picture}( 52.5700, 22.7000)(  4.0000,-22.7000)
%
{\color[named]{Black}{%
\special{pn 8}%
\special{pa 3218 1810}%
\special{pa 5588 1110}%
\special{dt 0.045}%
}}%
\put(43.0700,-17.9000){\makebox(0,0)[lb]{{\small $W_1$}}}%
%
{\color[named]{Black}{%
\special{pn 8}%
\special{pa 3228 1810}%
\special{pa 5618 1810}%
\special{fp}%
\special{sh 1}%
\special{pa 5618 1810}%
\special{pa 5550 1790}%
\special{pa 5564 1810}%
\special{pa 5550 1830}%
\special{pa 5618 1810}%
\special{fp}%
}}%
%
{\color[named]{Black}{%
\special{pn 13}%
\special{pa 3898 1610}%
\special{pa 4898 1310}%
\special{fp}%
}}%
%
{\color[named]{Black}{%
\special{pn 13}%
\special{pa 3698 1810}%
\special{pa 3726 1796}%
\special{pa 3754 1780}%
\special{pa 3778 1760}%
\special{pa 3800 1734}%
\special{pa 3814 1706}%
\special{pa 3826 1680}%
\special{pa 3846 1652}%
\special{pa 3882 1620}%
\special{pa 3898 1610}%
\special{pa 3888 1620}%
\special{sp}%
}}%
%
{\color[named]{Black}{%
\special{pn 13}%
\special{pa 4898 1310}%
\special{pa 4924 1330}%
\special{pa 4946 1350}%
\special{pa 4964 1374}%
\special{pa 4976 1402}%
\special{pa 4984 1432}%
\special{pa 4990 1466}%
\special{pa 4994 1500}%
\special{pa 4994 1534}%
\special{pa 4996 1568}%
\special{pa 4998 1604}%
\special{pa 5000 1638}%
\special{pa 5004 1670}%
\special{pa 5012 1702}%
\special{pa 5022 1732}%
\special{pa 5038 1758}%
\special{pa 5060 1784}%
\special{pa 5086 1804}%
\special{pa 5118 1812}%
\special{pa 5138 1810}%
\special{sp}%
}}%
%
{\color[named]{Black}{%
\special{pn 8}%
\special{pa 3228 1800}%
\special{pa 3228 360}%
\special{fp}%
\special{sh 1}%
\special{pa 3228 360}%
\special{pa 3208 428}%
\special{pa 3228 414}%
\special{pa 3248 428}%
\special{pa 3228 360}%
\special{fp}%
}}%
%
{\color[named]{Black}{%
\special{pn 8}%
\special{pa 3878 1610}%
\special{pa 3878 1810}%
\special{dt 0.045}%
}}%
%
{\color[named]{Black}{%
\special{pn 8}%
\special{pa 4898 1330}%
\special{pa 4888 1810}%
\special{dt 0.045}%
}}%
%
{\color[named]{Black}{%
\special{pn 13}%
\special{pa 5138 1810}%
\special{pa 5568 1800}%
\special{fp}%
}}%
%
{\color[named]{Black}{%
\special{pn 13}%
\special{pa 3228 1810}%
\special{pa 3698 1800}%
\special{fp}%
}}%
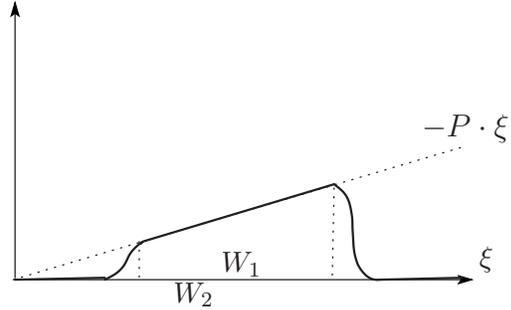
\put(33.2000,-23.9000){\makebox(0,0)[lb]{Fig. 4.2. The graph of $\varphi$ in case $n=1$.}}%
\put(56.5700,-17.7000){\makebox(0,0)[lb]{{\small$\xi$}}}%
\put(53.5700,-10.9000){\makebox(0,0)[lb]{$-P \cdot \xi$}}%
\put(40.5700,-19.5000){\makebox(0,0)[lb]{{\small $W_2$}}}%
%
{\color[named]{Black}{%
\special{pn 8}%
\special{pa 400 1800}%
\special{pa 2600 1800}%
\special{fp}%
\special{sh 1}%
\special{pa 2600 1800}%
\special{pa 2534 1780}%
\special{pa 2548 1800}%
\special{pa 2534 1820}%
\special{pa 2600 1800}%
\special{fp}%
}}%
%
{\color[named]{Black}{%
\special{pn 8}%
\special{pa 600 2000}%
\special{pa 600 0}%
\special{fp}%
\special{sh 1}%
\special{pa 600 0}%
\special{pa 580 68}%
\special{pa 600 54}%
\special{pa 620 68}%
\special{pa 600 0}%
\special{fp}%
}}%
%
{\color[named]{Black}{%
\special{pn 8}%
\special{pa 2400 2000}%
\special{pa 2400 0}%
\special{dt 0.045}%
}}%
%
{\color[named]{Black}{%
\special{pn 8}%
\special{pa 400 200}%
\special{pa 2600 200}%
\special{dt 0.045}%
}}%
\put(5.2000,-24.0000){\makebox(0,0)[lb]{Fig. 4.1. The figure of $U_i, W_i$.}}%
%
{\color[named]{Black}{%
\special{pn 8}%
\special{pa 1532 802}%
\special{pa 1550 830}%
\special{pa 1566 858}%
\special{pa 1580 886}%
\special{pa 1594 914}%
\special{pa 1608 944}%
\special{pa 1618 974}%
\special{pa 1628 1004}%
\special{pa 1642 1066}%
\special{pa 1646 1098}%
\special{pa 1648 1130}%
\special{pa 1648 1162}%
\special{pa 1644 1194}%
\special{pa 1634 1224}%
\special{pa 1620 1254}%
\special{pa 1602 1280}%
\special{pa 1576 1300}%
\special{pa 1548 1312}%
\special{pa 1516 1316}%
\special{pa 1484 1316}%
\special{pa 1452 1310}%
\special{pa 1422 1300}%
\special{pa 1392 1286}%
\special{pa 1364 1272}%
\special{pa 1336 1254}%
\special{pa 1312 1236}%
\special{pa 1286 1216}%
\special{pa 1262 1194}%
\special{pa 1240 1172}%
\special{pa 1218 1148}%
\special{pa 1178 1098}%
\special{pa 1160 1072}%
\special{pa 1142 1046}%
\special{pa 1126 1018}%
\special{pa 1110 990}%
\special{pa 1094 962}%
\special{pa 1068 904}%
\special{pa 1058 874}%
\special{pa 1050 842}%
\special{pa 1042 812}%
\special{pa 1036 780}%
\special{pa 1032 748}%
\special{pa 1032 716}%
\special{pa 1034 684}%
\special{pa 1038 652}%
\special{pa 1048 622}%
\special{pa 1064 594}%
\special{pa 1086 570}%
\special{pa 1112 552}%
\special{pa 1142 542}%
\special{pa 1174 540}%
\special{pa 1206 542}%
\special{pa 1238 550}%
\special{pa 1268 560}%
\special{pa 1296 574}%
\special{pa 1324 590}%
\special{pa 1352 608}%
\special{pa 1376 628}%
\special{pa 1402 648}%
\special{pa 1424 670}%
\special{pa 1448 692}%
\special{pa 1468 718}%
\special{pa 1488 742}%
\special{pa 1508 768}%
\special{pa 1526 794}%
\special{pa 1532 802}%
\special{sp}%
}}%
%
{\color[named]{Black}{%
\special{pn 8}%
\special{pa 1982 1032}%
\special{pa 1990 1064}%
\special{pa 1988 1096}%
\special{pa 1976 1126}%
\special{pa 1962 1154}%
\special{pa 1944 1182}%
\special{pa 1924 1206}%
\special{pa 1902 1228}%
\special{pa 1878 1250}%
\special{pa 1854 1272}%
\special{pa 1830 1292}%
\special{pa 1802 1310}%
\special{pa 1776 1328}%
\special{pa 1748 1344}%
\special{pa 1722 1360}%
\special{pa 1692 1376}%
\special{pa 1664 1388}%
\special{pa 1634 1402}%
\special{pa 1606 1414}%
\special{pa 1576 1424}%
\special{pa 1544 1436}%
\special{pa 1514 1444}%
\special{pa 1484 1452}%
\special{pa 1452 1460}%
\special{pa 1420 1466}%
\special{pa 1388 1470}%
\special{pa 1356 1474}%
\special{pa 1324 1474}%
\special{pa 1292 1474}%
\special{pa 1262 1472}%
\special{pa 1230 1466}%
\special{pa 1200 1454}%
\special{pa 1172 1438}%
\special{pa 1148 1416}%
\special{pa 1134 1388}%
\special{pa 1132 1356}%
\special{pa 1138 1324}%
\special{pa 1150 1294}%
\special{pa 1164 1266}%
\special{pa 1184 1240}%
\special{pa 1206 1216}%
\special{pa 1228 1194}%
\special{pa 1252 1172}%
\special{pa 1276 1152}%
\special{pa 1302 1132}%
\special{pa 1382 1082}%
\special{pa 1410 1066}%
\special{pa 1438 1050}%
\special{pa 1468 1038}%
\special{pa 1528 1014}%
\special{pa 1558 1002}%
\special{pa 1588 992}%
\special{pa 1618 984}%
\special{pa 1680 970}%
\special{pa 1712 964}%
\special{pa 1744 960}%
\special{pa 1776 958}%
\special{pa 1808 956}%
\special{pa 1840 958}%
\special{pa 1872 962}%
\special{pa 1902 970}%
\special{pa 1932 982}%
\special{pa 1958 1000}%
\special{pa 1978 1026}%
\special{pa 1982 1032}%
\special{sp}%
}}%
%
{\color[named]{Black}{%
\special{pn 8}%
\special{pa 1982 1032}%
\special{pa 1982 1032}%
\special{sp -0.045}%
}}%
%
{\color[named]{Black}{%
\special{pn 8}%
\special{pa 1100 480}%
\special{pa 1168 472}%
\special{pa 1202 468}%
\special{pa 1236 466}%
\special{pa 1268 462}%
\special{pa 1302 462}%
\special{pa 1334 460}%
\special{pa 1366 460}%
\special{pa 1396 462}%
\special{pa 1428 466}%
\special{pa 1458 470}%
\special{pa 1486 476}%
\special{pa 1514 484}%
\special{pa 1542 496}%
\special{pa 1568 508}%
\special{pa 1594 522}%
\special{pa 1620 538}%
\special{pa 1644 554}%
\special{pa 1668 574}%
\special{pa 1692 594}%
\special{pa 1714 616}%
\special{pa 1736 638}%
\special{pa 1760 662}%
\special{pa 1782 688}%
\special{pa 1802 714}%
\special{pa 1824 740}%
\special{pa 1868 796}%
\special{pa 1890 826}%
\special{pa 1912 856}%
\special{pa 1934 884}%
\special{pa 1956 916}%
\special{pa 1978 946}%
\special{pa 2000 976}%
\special{pa 2042 1036}%
\special{pa 2060 1068}%
\special{pa 2078 1096}%
\special{pa 2092 1126}%
\special{pa 2106 1154}%
\special{pa 2116 1182}%
\special{pa 2124 1210}%
\special{pa 2130 1236}%
\special{pa 2132 1260}%
\special{pa 2130 1284}%
\special{pa 2124 1308}%
\special{pa 2116 1328}%
\special{pa 2102 1348}%
\special{pa 2086 1368}%
\special{pa 2068 1386}%
\special{pa 2046 1402}%
\special{pa 2022 1418}%
\special{pa 1994 1434}%
\special{pa 1966 1448}%
\special{pa 1934 1462}%
\special{pa 1902 1474}%
\special{pa 1868 1486}%
\special{pa 1832 1498}%
\special{pa 1794 1508}%
\special{pa 1756 1520}%
\special{pa 1718 1530}%
\special{pa 1598 1556}%
\special{pa 1556 1566}%
\special{pa 1476 1582}%
\special{pa 1438 1590}%
\special{pa 1398 1596}%
\special{pa 1360 1604}%
\special{pa 1322 1608}%
\special{pa 1286 1614}%
\special{pa 1252 1618}%
\special{pa 1218 1620}%
\special{pa 1184 1620}%
\special{pa 1154 1620}%
\special{pa 1124 1618}%
\special{pa 1096 1614}%
\special{pa 1070 1606}%
\special{pa 1046 1598}%
\special{pa 1024 1588}%
\special{pa 1004 1574}%
\special{pa 988 1558}%
\special{pa 972 1538}%
\special{pa 960 1518}%
\special{pa 948 1494}%
\special{pa 938 1470}%
\special{pa 932 1442}%
\special{pa 924 1412}%
\special{pa 920 1382}%
\special{pa 916 1350}%
\special{pa 914 1316}%
\special{pa 912 1280}%
\special{pa 910 1244}%
\special{pa 908 1208}%
\special{pa 908 1170}%
\special{pa 906 1134}%
\special{pa 906 1094}%
\special{pa 904 1056}%
\special{pa 904 1018}%
\special{pa 902 980}%
\special{pa 898 940}%
\special{pa 896 904}%
\special{pa 892 866}%
\special{pa 888 830}%
\special{pa 886 794}%
\special{pa 886 760}%
\special{pa 886 726}%
\special{pa 888 694}%
\special{pa 894 666}%
\special{pa 902 638}%
\special{pa 912 612}%
\special{pa 926 588}%
\special{pa 944 566}%
\special{pa 968 548}%
\special{pa 994 530}%
\special{pa 1022 516}%
\special{pa 1052 500}%
\special{pa 1084 488}%
\special{pa 1100 480}%
\special{sp -0.045}%
}}%
%
{\color[named]{Black}{%
\special{pn 8}%
\special{pa 1000 400}%
\special{pa 1094 374}%
\special{pa 1126 366}%
\special{pa 1156 360}%
\special{pa 1188 354}%
\special{pa 1218 350}%
\special{pa 1250 346}%
\special{pa 1280 346}%
\special{pa 1312 346}%
\special{pa 1342 348}%
\special{pa 1374 352}%
\special{pa 1404 356}%
\special{pa 1466 370}%
\special{pa 1496 380}%
\special{pa 1526 390}%
\special{pa 1556 402}%
\special{pa 1586 416}%
\special{pa 1618 430}%
\special{pa 1648 444}%
\special{pa 1678 462}%
\special{pa 1706 480}%
\special{pa 1736 498}%
\special{pa 1766 518}%
\special{pa 1796 540}%
\special{pa 1824 562}%
\special{pa 1854 584}%
\special{pa 1882 608}%
\special{pa 1912 634}%
\special{pa 1938 660}%
\special{pa 1966 686}%
\special{pa 1992 714}%
\special{pa 2018 740}%
\special{pa 2044 770}%
\special{pa 2068 798}%
\special{pa 2090 828}%
\special{pa 2112 858}%
\special{pa 2132 888}%
\special{pa 2152 918}%
\special{pa 2168 948}%
\special{pa 2184 980}%
\special{pa 2200 1010}%
\special{pa 2212 1040}%
\special{pa 2222 1072}%
\special{pa 2230 1102}%
\special{pa 2238 1134}%
\special{pa 2242 1164}%
\special{pa 2244 1194}%
\special{pa 2244 1224}%
\special{pa 2240 1252}%
\special{pa 2236 1282}%
\special{pa 2228 1310}%
\special{pa 2216 1338}%
\special{pa 2202 1366}%
\special{pa 2186 1392}%
\special{pa 2168 1418}%
\special{pa 2146 1442}%
\special{pa 2124 1466}%
\special{pa 2100 1488}%
\special{pa 2074 1508}%
\special{pa 2046 1528}%
\special{pa 2018 1546}%
\special{pa 1990 1562}%
\special{pa 1962 1576}%
\special{pa 1934 1588}%
\special{pa 1904 1600}%
\special{pa 1876 1610}%
\special{pa 1846 1620}%
\special{pa 1784 1636}%
\special{pa 1752 1642}%
\special{pa 1720 1650}%
\special{pa 1688 1656}%
\special{pa 1656 1662}%
\special{pa 1622 1668}%
\special{pa 1588 1674}%
\special{pa 1554 1682}%
\special{pa 1520 1688}%
\special{pa 1484 1696}%
\special{pa 1450 1702}%
\special{pa 1414 1708}%
\special{pa 1378 1714}%
\special{pa 1344 1718}%
\special{pa 1310 1722}%
\special{pa 1276 1724}%
\special{pa 1242 1726}%
\special{pa 1210 1726}%
\special{pa 1178 1724}%
\special{pa 1148 1720}%
\special{pa 1118 1716}%
\special{pa 1090 1708}%
\special{pa 1062 1698}%
\special{pa 1036 1686}%
\special{pa 1012 1672}%
\special{pa 990 1656}%
\special{pa 970 1636}%
\special{pa 950 1614}%
\special{pa 932 1590}%
\special{pa 916 1564}%
\special{pa 902 1536}%
\special{pa 888 1506}%
\special{pa 874 1476}%
\special{pa 862 1444}%
\special{pa 850 1410}%
\special{pa 840 1376}%
\special{pa 830 1342}%
\special{pa 820 1308}%
\special{pa 794 1204}%
\special{pa 778 1134}%
\special{pa 772 1100}%
\special{pa 764 1066}%
\special{pa 754 998}%
\special{pa 752 964}%
\special{pa 748 930}%
\special{pa 748 898}%
\special{pa 746 866}%
\special{pa 748 834}%
\special{pa 750 802}%
\special{pa 754 772}%
\special{pa 760 742}%
\special{pa 768 712}%
\special{pa 776 684}%
\special{pa 788 656}%
\special{pa 800 630}%
\special{pa 816 602}%
\special{pa 832 578}%
\special{pa 850 552}%
\special{pa 870 528}%
\special{pa 892 504}%
\special{pa 938 458}%
\special{pa 986 412}%
\special{pa 1000 400}%
\special{sp -0.045}%
}}%
\put(11.5000,-8.3000){\makebox(0,0)[lb]{{\small $U_1$}}}%
\put(17.2000,-11.9000){\makebox(0,0)[lb]{{\small $U_2$}}}%
%
{\color[named]{Black}{%
\special{pn 8}%
\special{pa 1630 1010}%
\special{pa 1370 1270}%
\special{fp}%
\special{pa 1640 1060}%
\special{pa 1410 1290}%
\special{fp}%
\special{pa 1650 1110}%
\special{pa 1450 1310}%
\special{fp}%
\special{pa 1640 1180}%
\special{pa 1500 1320}%
\special{fp}%
\special{pa 1580 1000}%
\special{pa 1330 1250}%
\special{fp}%
\special{pa 1490 1030}%
\special{pa 1300 1220}%
\special{fp}%
\special{pa 1360 1100}%
\special{pa 1270 1190}%
\special{fp}%
}}%
\put(16.6000,-9.1000){\makebox(0,0)[lb]{{\small $W_1$}}}%
\put(18.6000,-5.1000){\makebox(0,0)[lb]{{\small $W_2$}}}%
\end{picture}%
\medskip
\end{center}

\begin{proof}
Without loss of generality, we may assume that $q=0$. 
Take $\xi_0 \in U_1 \cap U_2$. 
By Proposition \ref{prop:rep-2} we have 
\begin{align*}
\ol{H}(P) 
\ge&\, 
\inf_{(\varphi_1,\varphi_2) \in C^1(\T^n)}
\max_{i=1,2}\,
[|P+D\varphi_i(\xi_0)|^2-V_i(\xi_0)+\varphi_i(\xi_0,P)-\varphi_j(\xi_0,P)]\\
\ge&\, 
\inf_{(\varphi_1,\varphi_2) \in C^1(\T^n)}
\max_{i=1,2}\,
[\varphi_i(\xi_0,P)-\varphi_j(\xi_0,P)]=0.
\end{align*}

Now, let 
$d:=\min\{ {\rm dist}(W_1,\partial W_2), \ {\rm dist}(U_1 \cup U_2, \partial W_1)\}>0$.
There exists $\ep_0>0$ such that
\begin{equation}\label{flat:eqn1}
V_i(\xi) \ge \ep_0 \quad \text{for } x \in \T^n \setminus W_1,
\ i=1,2.
\end{equation}
We define a smooth function $\varphi$ on $\T^n$ such that
\begin{align}
&\varphi(\xi)= - P \cdot \xi \ \text{on } W_1, \
\varphi(\xi)=0\ \text{on } \T^n \setminus W_2, \nonumber\\
& |D\varphi| \le \dfrac{C|P|}{d} \ \text{on } \T^n. 
\label{sub-phi}
\end{align}
Notice that
$$
|P+D\varphi(\xi)|^2-V_i(\xi)=
\begin{cases}
=-V_i(\xi) \le 0, \qquad&\text{on } W_1,\\
\le \dfrac{C|P|^2}{d^2}-\ep_0, &\text{on } \T^n \setminus W_1.
\end{cases}
$$
Thus, $|P+D\varphi(\xi)|^2-V_i(\xi) \le 0$ on $\T^n$ provided
that $|P| \le d \sqrt{\ep_0}/C=:\gam$.
We hence have that $(\varphi,\varphi,0)$ is a subsolution
of $({\rm E}_P)$ for $|P| \le \gam$.
Therefore $\ol{H}(P) \le 0$ for $|P| \le \gam$ by Proposition \ref{prop:rep-1}.
\end{proof}

\begin{rem} \label{flat:rem1}
(i) In fact, the result of Theorem \ref{thm:main4} still holds
for more general Hamiltonians 
$$
H_i(\xi,p):=F_i(\xi,p)-V_i(\xi), 
$$
where $F_i\in C(\T^n\times\R^n)$ and $V_i\in C(\T^n)$ 
are assumed to satisfy 
\begin{itemize}
\item[{\rm (a)}] 
the functions $p\mapsto F_{i}(\xi,p)$ are convex and 
$F_{i}(\xi,p)\ge F_{i}(\xi,0)=0$
for all $(\xi,p)\in\T^n\times\R^{n}$, 
\item[{\rm (b)}] 
$V_1,\, V_2$ satisfy the conditions of Theorem \ref{thm:main4}.
\end{itemize}
(ii) Notice that the assumptions $V_i \ge 0$ and 
{$\{V_i=0\}\not=\emptyset$ }
for $i=1,2$ are
just for simplicity. In general, we can normalize $V_i$ by
$V_i - \min_{\xi\in \T^n} V_i(\xi)$ to get back to such situation.\\
(iii) From the proof of Theorem \ref{thm:main4} 
we have 
$$
v_1(\xi,P)=v_2(\xi,P) \quad \text{for all} \  \xi \in \{V_1=0\} \cap \{V_2=0\}
$$
for any solution $(v_1(\cdot,P),v_2(\cdot,P))$ of $({\rm E}_P)$.\\
(iv) By Proposition \ref{prop:max-H1H2} we can give another proof 
to Theorem \ref{thm:main4} as follows. 
In this case, we explicitly have
$$
K(\xi,p)=\max \{ |p|^2-V_1(\xi), |p|^2-V_2(\xi)\} = |p|^2 -V(\xi)
$$
where $V(\xi)=\min \{V_1(\xi), V_2(\xi)\}$.
Note that $V \ge 0$ and $\{V=0\}= \{V_1=0\} \cup \{V_2=0\}$.
Hence, we can either repeat the above proof for single equations
to show that $\ol{H}(P) \le \ol{K}(P) = 0$ for $|P| \le \gam$
or we can use Concordel's result directly.
\end{rem}


Notice that condition \eqref{flat.con} is crucial and 
plays an important role in the construction of 
the subsolution $(\varphi,\varphi,0)$ of $(\textrm{E}_P)$ 
and could not be removed in the proof of Theorem \ref{thm:main4}.
We point out in the next Theorem that
there are cases when \eqref{flat.con} does not hold, then the
flatness near the origin of $\ol{H}$ does not appear.

\begin{thm} \label{thm:main5}
Assume that $V_i \ge 0$ in $\T^n$ and 
\begin{equation} \label{not.flat}
\{V_1=0\}=\{V_2=0\}=\{\xi=(\xi_1,\cdots,\xi_n) \in \T^n:\ \xi_j=1/2 
\ \text{for } j \ge 2\}=:K. 
\end{equation}
The following hold. 
\begin{itemize}
\item[(i)] 
There exists $\gam>0$ such that 
$\ol{H}(P)=|P_1|^2$ provided that 
$|P'| \le \gam$
for any $P=(P_1,P')\in\R\times\R^{n-1}$. 
\item[(ii)] 
$\ol{H}(P) \ge |P_1|^2$ 
for all $P \in \R^n$.  

\end{itemize}
\end{thm}
\begin{proof}
Firstly, we prove that $\ol{H}(P) \le |P_1|^2$ provided that
$|P'| \le \gam$ for some $\gam>0$ small enough by using
exactly the same idea in the proof of Theorem \ref{thm:main4}.
We build a function $\varphi(\xi)=\varphi(\xi_2,\cdots,\xi_n)
{\in C^{1}(\T^n)}$, which
does not depend on $\xi_1$, so that
$$
\sum_{j=2}^n |P_j+\varphi_{\xi_j}(\xi)|^2 - V_i(\xi) \le 0\ \text{on } \T^n
$$
for $i=1,2$ and for $|P'| \le \gam$ with $\gam>0$ small enough.
Thus
$$
|P+D\varphi(\xi)|^2 - V_i(\xi) \le |P_1|^2 \ \text{on } \T^n
$$
for $i=1,2$.  
By Proposition \ref{prop:rep-1},  
$\ol{H}(P) \le |P_1|^2$. 

We now prove that $\ol{H}(P) \ge |P_1|^2$.
For each $\xi_0\in K$, we have in view of \eqref{rep-implicit}
\begin{align*}
\ol{H}(P) &= \inf_{\varphi \in C^1(\T^n)} \max_{\xi \in \T^n}
[|P+D\varphi(\xi)|^2-V_1(\xi)+v_1(\xi,P)-v_2(\xi,P)]\\
&\ge \inf_{\varphi \in C^1(\T^n)}
[|P+D\varphi(\xi_0)|^2+v_1(\xi_0,P)-v_2(\xi_0,P)],
\end{align*}
and similarly
\begin{align*}
\ol{H}(P) \ge \inf_{\psi \in C^1(\T^n)}
[|P+D\psi(\xi_0)|^2+v_2(\xi_0,P)-v_1(\xi_0,P)].
\end{align*}
Take an arbitrary function $\varphi \in C^1(\T^n)$ and observe that
\begin{align*}
&\int_K |P+D\varphi(\xi)|^2\, d\xi_1\\
\ge& \int_K |P_1+\varphi_{\xi_1}(\xi)|^2 \, d\xi_1
=\int_K |P_1|^2+|\varphi_{\xi_1}(\xi)|^2 + 2 P_1 \varphi_{\xi_1}(\xi)\, d \xi_1\\
\ge & \int_K |P_1|^2 + 2 P_1 \varphi_{\xi_1}(\xi)\,d \xi_1 
= |P_1|^2.
\end{align*}
Thus, it is clear to see that $\ol{H}(P) \ge |P_1|^2$, which implies the result.
\end{proof}

The above two Theorems 
describe several examples that we can obtain similar results 
of the flat part or non-flat part of $\ol{H}$
to those of single Hamilton--Jacobi equations   
in \cite{C1,C2}. 
Indeed, the structures on the potentials $V_i$ for $i=1,2$ are 
very related in such a way that we obtain 
the {shape} of $\ol{H}$ 
like for single equations.
We rely on the idea of building the subsolutions
$(\varphi,\psi,\ol{H}(P))$ of $\EP$ where $\varphi=\psi$,
which does not work in general cases.

Next, we start investigating the properties of $\ol{H}$ in some cases where 
the structures of the potentials $V_i$ for $i=1,2$ are widely different and
in general we cannot expect $\ol{H}$ to have simple properties.
The next question is that: Can we read of information
of the effective Hamiltonian in the case where 
$\{V_1=\min_{\xi\in \T^n} V_1(\xi)\} \cap \{V_2=\min_{\xi\in \T^n} V_2(\xi)\}
=\emptyset$?

\begin{thm} \label{thm:main6}
Let $n=1$ and assume that for $\ep_0>0$ small enough 
the following properties hold. 
\begin{itemize}
\item[(a)]
$\{V_1=0\}=[\dfrac{4}{16},\dfrac{12}{16}]$, 
$\{V_1=-\ep_0\}=[0,1] \setminus (\dfrac{3}{16},\dfrac{13}{16})$,
and $-\ep_0 \le V_1 \le 0$ on $\T$
for some $\ep_0>0$.
\item[(b)] $\{V_2=0\}=[\dfrac{7}{16},\dfrac{9}{16}]$,
$\{V_2=2\}=[0,1] \setminus (\dfrac{6}{16},\dfrac{10}{16})$,
and $0 \le V_2 \le 2$ on $\T$.
\end{itemize}
There exists $\gam>0$ such that
$\ol{H}(P)=0$ 
for $|P| \le \gam$. 
\end{thm}

\begin{lem} \label{flat:lem1}
We have
$$
\ol{H}(P) \ge -\dfrac{1}{2} \min_{\xi \in \T^n} (V_1+V_2)(\xi).
$$
\end{lem}

\begin{proof}
Sum up the two equations in $({\rm E}_P)$ to get
$$
|P+Dv_1|^2+|P+Dv_2|^2 - V_1-V_2= 2 \ol{H}(P), 
$$
which implies $2\ol{H}(P) \ge -(V_1+V_2)(\xi))$ for a.e. $\xi \in \T^n$, 
and the proof is complete.
\end{proof}

\begin{proof}[Proof of Theorem {\rm\ref{thm:main6}}]
Noting that $\min_{\xi \in \T} (V_1+V_2)(\xi)=0$, 
and $\{V_1=-\ep_0\} \cap \{V_2=0\} = \emptyset$, 
we have $\ol{H}(P) \ge 0$ by Lemma \ref{flat:lem1}.
We construct a subsolution $(\varphi, \psi,0)$ of 
$\EP$ for $|P|$ small enough. 
Let
$$
W_1=(\dfrac{6}{16},\dfrac{10}{16}),\
W_2=(\dfrac{5}{16},\dfrac{11}{16}),\
W_3=(\dfrac{4}{16},\dfrac{12}{16}).
$$
Let  $P<0$ for simplicity.
We define the functions $\varphi, \psi$ by 
$$
\varphi(\xi) :=
\begin{cases}
-P\cdot \xi \qquad &\text{for} \ x\in W_2\\
0 &\text{for} \ \xi \in \T \setminus W_3
\end{cases}
$$
and $|D\varphi| \le C_1|P|$ for some $C_1>0$, 
$0\le \varphi \le -P\cdot \xi$ on $[0,1]$ and  
$$
\psi(\xi) =
\begin{cases}
-P\cdot \xi \qquad &\text{for} \ \xi \in W_1\\
C_2 &\xi \in \T \setminus W_2
\end{cases}
$$
for some $C_2 \in (1/64,1)$, 
$|P+D\psi| \le 1$, and $\psi \ge -P\cdot \xi$ on $[0,1]$.

We have
\begin{align*}
&|P+D\varphi(\xi)|^2-V_1(\xi)+\varphi(\xi)-\psi(\xi)\\
\le&\, 
\begin{cases}
\varphi(\xi)-\psi(\xi) \le 0 \qquad 
&\text{if} \ \xi \in W_2\\
{2(C_1^2+1)}|P|^2+\ep_0+|P| - C_2 
&\text{if} \ \xi \in \T \setminus W_2.
\end{cases}
\end{align*}
If $|P|$ and $\ep_0$ are small enough, then 
$|P+D\varphi(\xi)|^2-V_1(\xi)+\varphi(\xi)-\psi(\xi) \le 0$
on $\T$.
Besides,
\begin{align*}
&|P+D\psi(\xi)|^2-V_2(\xi)+\psi(\xi)-\varphi(\xi)\\
\le &\, 
\begin{cases}
0 \qquad \qquad 
&\text{if} \ \xi\in W_1\\
1-2+C_2-0 \le 0 
&\text{if} \ \xi \in \T \setminus W_1.
\end{cases}
\end{align*}
\vspace{13pt}

\begin{center}
\unitlength 0.1in
\begin{picture}( 60.1700, 17.2000)(  1.0300,-20.7000)
%
{\color[named]{Black}{%
\special{pn 8}%
\special{pa 130 1800}%
\special{pa 2520 1800}%
\special{fp}%
\special{sh 1}%
\special{pa 2520 1800}%
\special{pa 2454 1780}%
\special{pa 2468 1800}%
\special{pa 2454 1820}%
\special{pa 2520 1800}%
\special{fp}%
}}%
%
{\color[named]{Black}{%
\special{pn 8}%
\special{pa 3730 1800}%
\special{pa 6120 1800}%
\special{fp}%
\special{sh 1}%
\special{pa 6120 1800}%
\special{pa 6054 1780}%
\special{pa 6068 1800}%
\special{pa 6054 1820}%
\special{pa 6120 1800}%
\special{fp}%
}}%
%
{\color[named]{Black}{%
\special{pn 8}%
\special{pa 120 1800}%
\special{pa 2490 1100}%
\special{dt 0.045}%
}}%
%
{\color[named]{Black}{%
\special{pn 8}%
\special{pa 3730 1800}%
\special{pa 6100 1100}%
\special{dt 0.045}%
}}%
%
{\color[named]{Black}{%
\special{pn 13}%
\special{pa 800 1600}%
\special{pa 1800 1300}%
\special{fp}%
}}%
%
{\color[named]{Black}{%
\special{pn 13}%
\special{pa 600 1800}%
\special{pa 630 1786}%
\special{pa 658 1770}%
\special{pa 682 1750}%
\special{pa 702 1724}%
\special{pa 718 1696}%
\special{pa 730 1670}%
\special{pa 750 1642}%
\special{pa 786 1610}%
\special{pa 800 1600}%
\special{pa 790 1610}%
\special{sp}%
}}%
%
{\color[named]{Black}{%
\special{pn 13}%
\special{pa 1800 1300}%
\special{pa 1828 1320}%
\special{pa 1850 1340}%
\special{pa 1866 1364}%
\special{pa 1880 1392}%
\special{pa 1888 1422}%
\special{pa 1894 1456}%
\special{pa 1896 1490}%
\special{pa 1898 1524}%
\special{pa 1898 1558}%
\special{pa 1900 1594}%
\special{pa 1902 1628}%
\special{pa 1908 1660}%
\special{pa 1914 1692}%
\special{pa 1926 1722}%
\special{pa 1942 1748}%
\special{pa 1964 1774}%
\special{pa 1990 1794}%
\special{pa 2020 1802}%
\special{pa 2040 1800}%
\special{sp}%
}}%
%
{\color[named]{Black}{%
\special{pn 13}%
\special{pa 3730 800}%
\special{pa 4400 800}%
\special{fp}%
}}%
%
{\color[named]{Black}{%
\special{pn 13}%
\special{pa 4400 800}%
\special{pa 4432 810}%
\special{pa 4462 820}%
\special{pa 4486 838}%
\special{pa 4506 858}%
\special{pa 4522 884}%
\special{pa 4536 914}%
\special{pa 4548 944}%
\special{pa 4556 978}%
\special{pa 4564 1012}%
\special{pa 4570 1048}%
\special{pa 4578 1084}%
\special{pa 4584 1118}%
\special{pa 4592 1150}%
\special{pa 4600 1182}%
\special{pa 4608 1214}%
\special{pa 4618 1244}%
\special{pa 4630 1274}%
\special{pa 4642 1302}%
\special{pa 4656 1330}%
\special{pa 4672 1358}%
\special{pa 4690 1384}%
\special{pa 4708 1410}%
\special{pa 4730 1434}%
\special{pa 4756 1454}%
\special{pa 4782 1472}%
\special{pa 4808 1490}%
\special{pa 4810 1490}%
\special{sp}%
}}%
%
{\color[named]{Black}{%
\special{pn 13}%
\special{pa 4810 1490}%
\special{pa 5000 1430}%
\special{fp}%
}}%
%
{\color[named]{Black}{%
\special{pn 13}%
\special{pa 4990 1430}%
\special{pa 5036 1386}%
\special{pa 5058 1362}%
\special{pa 5078 1338}%
\special{pa 5098 1312}%
\special{pa 5118 1286}%
\special{pa 5134 1260}%
\special{pa 5150 1230}%
\special{pa 5164 1202}%
\special{pa 5176 1172}%
\special{pa 5196 1112}%
\special{pa 5204 1082}%
\special{pa 5214 1050}%
\special{pa 5222 1020}%
\special{pa 5242 958}%
\special{pa 5254 928}%
\special{pa 5270 898}%
\special{pa 5288 874}%
\special{pa 5310 852}%
\special{pa 5336 832}%
\special{pa 5366 818}%
\special{pa 5394 804}%
\special{pa 5400 800}%
\special{sp}%
}}%
%
{\color[named]{Black}{%
\special{pn 13}%
\special{pa 5390 800}%
\special{pa 5980 810}%
\special{fp}%
}}%
%
{\color[named]{Black}{%
\special{pn 8}%
\special{pa 130 1790}%
\special{pa 130 350}%
\special{fp}%
\special{sh 1}%
\special{pa 130 350}%
\special{pa 110 418}%
\special{pa 130 404}%
\special{pa 150 418}%
\special{pa 130 350}%
\special{fp}%
}}%
%
{\color[named]{Black}{%
\special{pn 8}%
\special{pa 3730 1800}%
\special{pa 3730 370}%
\special{fp}%
\special{sh 1}%
\special{pa 3730 370}%
\special{pa 3710 438}%
\special{pa 3730 424}%
\special{pa 3750 438}%
\special{pa 3730 370}%
\special{fp}%
}}%
%
{\color[named]{Black}{%
\special{pn 8}%
\special{pa 780 1600}%
\special{pa 780 1800}%
\special{dt 0.045}%
}}%
%
{\color[named]{Black}{%
\special{pn 8}%
\special{pa 1800 1320}%
\special{pa 1790 1800}%
\special{dt 0.045}%
}}%
%
{\color[named]{Black}{%
\special{pn 8}%
\special{pa 4410 800}%
\special{pa 4410 1800}%
\special{dt 0.045}%
}}%
%
{\color[named]{Black}{%
\special{pn 8}%
\special{pa 5390 810}%
\special{pa 5400 1800}%
\special{dt 0.045}%
}}%
%
{\color[named]{Black}{%
\special{pn 8}%
\special{pa 4810 1480}%
\special{pa 4810 1800}%
\special{dt 0.045}%
}}%
%
{\color[named]{Black}{%
\special{pn 8}%
\special{pa 4990 1420}%
\special{pa 4990 1810}%
\special{dt 0.045}%
}}%
%
{\color[named]{Black}{%
\special{pn 13}%
\special{pa 2040 1800}%
\special{pa 2470 1790}%
\special{fp}%
}}%
%
{\color[named]{Black}{%
\special{pn 13}%
\special{pa 130 1800}%
\special{pa 600 1790}%
\special{fp}%
}}%
\put(3.2000,-22.0000){\makebox(0,0)[lb]{Fig. 4.3. The graph of $\varphi$.}}%
\put(39.0000,-21.9000){\makebox(0,0)[lb]{Fig. 4.4. The graph of $\psi$.}}%
\put(61.1000,-17.9000){\makebox(0,0)[lb]{{\small$\xi$}}}%
\put(25.0000,-17.8000){\makebox(0,0)[lb]{{\small$\xi$}}}%
\put(22.6000,-10.8000){\makebox(0,0)[lb]{$-P \cdot \xi$}}%
\put(12.1000,-17.8000){\makebox(0,0)[lb]{{\small $W_2$}}}%
\put(9.6000,-19.4000){\makebox(0,0)[lb]{{\small $W_3$}}}%
\put(46.7000,-19.5000){\makebox(0,0)[lb]{{\small $W_2$}}}%
\put(48.1000,-17.9000){\makebox(0,0)[lb]{{\small $W_1$}}}%
\put(34.9000,-8.5000){\makebox(0,0)[lb]{$C_2$}}%
\end{picture}%

\medskip
\end{center}

Thus $(\varphi,\psi,0)$ is a subsolution of $({\rm E}_P)$, 
and the proof is complete.
\end{proof}

\begin{rem}\label{flat:rem2}
It is worth to notice that
$$
\ol{H}(0) \ne -\dfrac{1}{2} \min_{\xi \in \T^n} (V_1+V_2)(\xi)
$$
in general.
Indeed, set 
\begin{align*}
&V_1(\xi)=4\pi^2 \sin^2 (2\pi \xi)+\cos(2\pi \xi) - \sin(2\pi \xi),\\
&V_2(\xi)=4\pi^2 \cos^2 (2\pi \xi)+\sin(2\pi \xi) - \cos(2\pi \xi). 
\end{align*}
Clearly 
$(\cos(2\pi \xi), \sin(2\pi \xi),0)$ is a solution of (E$_0$), 
and hence $\ol{H}(0)=0$. 
In this case
$$
\ol{H}(0)=0 \ne -2\pi^2=-\dfrac{1}{2}(V_1+V_2)(\xi) 
\ \text{for all} \ \xi \in \T.
$$
\end{rem}

In Theorem \ref{thm:main4} the fact that 
$\Pi(\R^n\setminus (U_1\cup U_2))$ is connected, 
where 
$U_i=\{V_i=0\}$
plays an important role in the construction of
subsolutions as stated just before 
Theorem \ref{thm:main5}. 
In the next couple of Theorems we make new observations that 
we can get the flat parts of effective Hamiltonians even though 
$\Pi(\R^n\setminus (U_1\cup U_2))$ is not connected.

\begin{thm}\label{thm:main7}
Let $n=1$ and assume 
$V_1 \equiv 0$, 
$V_2 \ge 0$ on $[0,1]$ and $\{V_2=0\}=\{1/2\}$.
Then there exists $\gam>0$ such that
$\ol{H}(P)=0$ 
for $|P| \le \gam$. 
\end{thm}

\begin{proof}[Sketch of Proof]
The proof is almost the same as the proof of Theorem \ref{thm:main6}
but let us present it here for the sake of clarity.
Since $\min_{\xi\in\T^n}(V_1+V_2)(\xi)=0$, we have 
$\ol{H}(P)\ge0$ by Lemma \ref{flat:lem1}. 
Let
$$
W_1=(\dfrac{3}{8},\dfrac{5}{8}),\
W_2= (\dfrac{2}{8}, \dfrac{6}{8}),\
W_3=(\dfrac{1}{8},\dfrac{7}{8}).
$$
There exists $M \in (0,1)$ so that
$$
V_2(\xi) \ge M \ \text{for} \ \xi \notin W_1.
$$
Assume $P<0$ for simplicity. We now
construct the functions $\varphi,\psi$ 
so that $(\varphi,\psi,0)$ is a subsolution of $\EP$ for 
small $|P|$, which implies the conclusion.
Take $|P| \le M/4$ first.
We define the functions $\varphi, \psi$ by 
$$
\varphi(\xi) :=
\begin{cases}
-P\cdot \xi \qquad &\text{for} \ x\in W_2\\
0 &\text{for} \ \xi \in \T \setminus W_3
\end{cases}
$$
and $|D\varphi| \le C_1|P|$ for some $C_1>0$, 
$0\le \varphi \le -P\cdot \xi$ on $[0,1]$ and  
$$
\psi(\xi) =
\begin{cases}
-P\cdot \xi \qquad &\text{for} \ \xi \in W_1\\
C_2 &\xi \in \T \setminus W_2
\end{cases}
$$
for some $C_2 \in (M/128,M/2)$, 
$|P+D\psi| \le M/2$, and $\psi \ge -P\cdot \xi$ on $[0,1]$.

We have
\begin{align*}
|P+D\varphi(\xi)|^2+\varphi(\xi)-\psi(\xi)
\le&\, 
\begin{cases}
\varphi(\xi)-\psi(\xi) \le 0 \qquad 
&\text{if} \ \xi \in W_2\\
{2(C_1^2+1)}|P|^2+|P| - C_2
&\text{if} \ \xi \in \T \setminus W_2.
\end{cases}
\end{align*}
If $|P|$ is small enough, then 
$|P+D\varphi(\xi)|^2+\varphi(\xi)-\psi(\xi) \le 0$
on $\T$.
Besides,
\begin{align*}
&|P+D\psi(\xi)|^2-V_2(\xi)+\psi(\xi)-\varphi(\xi)\\
\le &\, 
\begin{cases}
0 \qquad \qquad 
&\text{if} \ \xi\in W_1\\
\dfrac{M^2}{4}-M+C_2-0 \le \dfrac{M^2}{4}-M+\dfrac{M}{2} \le 0 
&\text{if} \ \xi \in \T \setminus W_1.
\end{cases}
\end{align*}
Thus $(\varphi,\psi,0)$ is a subsolution of $({\rm E}_P)$, 
and the proof is complete.
\end{proof}

%
%
%
%
%

We can actually generalize Theorem \ref{thm:main7} as following. 
\begin{thm}\label{thm:main7new}
Assume that $V_1 \equiv 0$, $V_2 \ge 0$ and there exist an open set
$W$ in $\T^n$  and 
a vector $q\in \R^n$ such that 
$\Pi(q+W) \Subset (0,1)^n$ and 
$\emptyset \ne \{V_2=0\} \subset W$.
Then there exists $\gam>0$ such that $\ol{H}(P)=0$ for $|P| \le \gam$.
\end{thm}
The proof of this Theorem is basically the same as the proof of Theorem
\ref{thm:main7}, hence omitted.
The following Corollary is a direct consequence of Theorem \ref{thm:main7new}
\begin{cor} \label{cor:main7}
Assume that $V_1,\, V_2 \ge 0$ and there exist an open set
$W$ in $\T^n$ and a vector $q\in \R^n$ such that 
$\Pi(q+W) \Subset (0,1)^n$ and 
$$
\emptyset \ne \{V_1=0\} \cap \{V_2=0\} \subset \{V_2=0\} \subset W.
$$
Then there exists $\gam>0$ such that $\ol{H}(P)=0$ for $|P| \le \gam$.
\end{cor}

The result of Corollary \ref{cor:main7} is pretty surprising in the sense
that flat part around $0$ of $\ol{H}$ occurs even though we do not know much
information about $V_1$.
More precisely, we only need to control well $\{V_2=0\}$ and do not need to care
about $\{V_1=0\}$ except that $\{V_1=0\} \cap \{V_2=0\} \ne \emptyset$.

Finally, we consider a situation in which the requirements of Theorem \ref{thm:main7new}
and Corollary \ref{cor:main7} fail.

\begin{thm} \label{thm:main8}
We take two potentials $V^i:\T \to [0,\infty)$ such that
$V^i$ are continuous and
$\{V^i=0\}=\{y_{0i}\}$ for some $y_{0i} \in \T$ for $i=1,2$.
Assume that $V_1(\xi_1,\xi_2)=V^1(\xi_1)$ and $V_2(\xi_1,\xi_2)=V^2(\xi_2)$
for $(\xi_1,\xi_2)\in \T^2$.
Then there exists $\gam>0$ such that
$\ol{H}(P)=0$ 
for $|P| \le \gam$. 
\end{thm}

\begin{center}
\unitlength 0.1in
\begin{picture}( 18.2000, 18.9000)( 23.7000,-24.9000)
\put(30.4000,-19.4000){\makebox(0,0)[lb]{{\small $\{V_1=0\}$}}}%
%
{\color[named]{Black}{%
\special{pn 8}%
\special{pa 2600 600}%
\special{pa 4190 600}%
\special{pa 4190 2200}%
\special{pa 2600 2200}%
\special{pa 2600 600}%
\special{pa 4190 600}%
\special{fp}%
}}%
%
{\color[named]{Black}{%
\special{pn 13}%
\special{pa 3210 600}%
\special{pa 3210 2200}%
\special{fp}%
}}%
%
{\color[named]{Black}{%
\special{pn 13}%
\special{pa 2600 1430}%
\special{pa 4190 1430}%
\special{fp}%
}}%
\put(31.2000,-23.2000){\makebox(0,0)[lb]{{\small $y_{01}$}}}%
\put(23.7000,-14.7000){\makebox(0,0)[lb]{{\small $y_{02}$}}}%
\put(33.2000,-14.0000){\makebox(0,0)[lb]{{\small $\{V_2=0\}$}}}%
\put(26.0000,-26.2000){\makebox(0,0)[lb]{Fig. 4.5. The figures of $\{V_i=0\}$}}%
\end{picture}%

\end{center}

\begin{proof}
By using Theorem \ref{thm:main7}, for $P=(P_1,P_2)$ with $|P|$ small enough,
there exist two pairs $(\varphi_i,\psi_i)\in C^{0,1}(\T)^2$ for $i=1,2$ such that
$$
\begin{cases}
|P_1+\varphi_1'(\xi_1)|^2-V^1(\xi_1)+\varphi_1(\xi_1)-\psi_1(\xi_1) = 0, \\
|P_1+\psi_1'(\xi_1)|^2+\psi_1(\xi_1)-\varphi_1(\xi_1)= 0
\end{cases}
$$
and
$$
\begin{cases}
|P_2+\varphi_2'(\xi_2)|^2+\varphi_2(\xi_2)-\psi_2(\xi_2) = 0, \\
|P_2+\psi_2'(\xi_2)|^2-V^2(\xi_2)+\psi_2(\xi_2)-\varphi_2(\xi_2)= 0
\end{cases}
$$

Now let $v_1(\xi_1,\xi_2)=\varphi_1(\xi_1)+\varphi_2(\xi_2)$,
$v_2(\xi_1,\xi_2)=\psi_1(\xi_1)+\psi_2(\xi_2)$
for $(\xi_1,\xi_2) \in \T^2$.
For $P=(P_1,P_2)$ with $|P| \le \gam$, we easily get
that $(v_1,v_2,0)$ is a solution of $(\textrm{E}_P)$, which means $\ol{H}(P)=0$.
\end{proof}


\section{Generalization}
In this section we consider weakly coupled systems 
of $m$-equations for $m \ge 2$ 
\[
(u_i^{\ep})_t+H_{i}(\frac{x}{\ep},Du_i^{\ep})
+\frac{1}{\ep}\sum_{j=1}^{m}c_{ij}(u_i^{\ep}-u_j^{\ep})=0 
\ \textrm{in} \ \Q 
\ \textrm{for} \ i=1,\ldots,m, 
\]
with 
\[
u_{i}^{\ep}(x,0)=f_i(x) \ \textrm{on} \ \R^n \ \textrm{for} \ i=1,\ldots,m, 
\]
where 
$c_{ij}$ are given nonnegative
constants which are assumed to 
satisfy 
\begin{equation}\label{ass:c_ij}
\sum_{j=1}^{m}c_{ij}=1 \ \textrm{for all} 
\ i=1,\ldots,m. 
\end{equation}
Set 
\[
K:=
{\small 
\left(
\begin{array}{ccc}
c_{11} &  \cdots & c_{1m} \\
\vdots & \ddots & \vdots\\
c_{m1} &  \cdots & c_{mm} 
\end{array}
\right), 
}
\
\mathbf{u^{\ep}}
:=
{\small
\left(
\begin{array}{c}
u_{1}^{\ep} \\
\vdots \\
u_{m}^{\ep}
\end{array}
\right), 
}
\
\text{and}
\ 
\mathbf{f}
:=
{\small
\left(
\begin{array}{c}
f_{1} \\
\vdots \\
f_{m}
\end{array}
\right). 
}
\]
Then the problem can be written as
\begin{numcases}
{}
\mathbf{u}^{\ep}_t
+
{\small
\left(
\begin{array}{c}
H_{1}(x/\ep,Du_{1}^{\ep}) \\
\vdots \\
H_{m}(x/\ep,Du_{m}^{\ep})
\end{array}
\right) 
}
+\frac{1}{\ep}(I-K)\mathbf{u^{\ep}}
= 0
& in $\Q$, \label{eq:m-system}\\
\mathbf{u^{\ep}}(\cdot,0)=\mathbf{f} 
& 
on $\R^{n}$, 
\nonumber
\end{numcases}
where $I$ is the 
identity matrix of size $m$.
We obtain the following result.
\begin{thm} \label{thm:main-m}
The functions $u^\ep_i$ converge locally uniformly to the same limit $u$
in $\Q$ as $\ep \to 0$ for $i=1,\ldots,m$ and $u$ solves
\[
\begin{cases}
u_t+\ol{H}(Du)=0  &\textrm{in } \Q \\
u(x,0) = \ol{f}(x) &\textrm{on } \R^n,
\end{cases}
\]
where $\ol{H}$ is the associated effective Hamiltonian and 
\[
\ol{f}(x):=\frac{1}{m}\sum_{i=1}^{m}f_i(x). 
\]
\end{thm}

We only present barrier functions which are generalizations of 
the barrier function in case $m=2$ defined by \eqref{func:barrier} 
in Lemma \ref{lem:barrier}. 
Set 
\[
\mathbf{w^{\pm}(x,t)}
:=
(\ol{f}\pm Ct)\mathbf{j}
+
\mathbf{g^{\ep}}(x,t), 
\]
where 
$C$ is a positive constant which will be fixed later, 
$\mathbf{j}:=(1,\ldots,1)^{T}$ and 
\begin{align*}
\mathbf{g^{\ep}}(x,t)
:=
\big[e^{\frac{t}{\ep}(K-I)}\mathbf{h}\big](x), \ 
\mathbf{h}(x):=\mathbf{f}(x)-\ol{f}(x)\mathbf{j}. 
\end{align*}

Since we assume \eqref{ass:c_ij}, we can easily check that 
the Frobenius root of $K$, i.e., the maximum of the eigenvalues 
of $K$, is $1$ and moreover 
$\mathbf{j}$ is an associated eigenvector. 
Moreover by the Perron--Frobenius theorem we have 
\begin{lem}
There exists $\del>0$ such that 
$|e^{t(K-I)}\mathbf{h}|\le e^{-\del t}|\mathbf{h}|$  
provided that $\mathbf{h} \cdot \mathbf{j}=0$. 
\end{lem}
See \cite[Lemma 5.2]{IS} for a more general result.

\begin{prop}
The functions $\mathbf{w^{\pm}}$ are a subsolution and a supersolution 
of \eqref{eq:m-system} with $\mathbf{w}^{\pm}(\cdot,0)=\mathbf{f}$ 
on $\R^n$, respectively,  if $C>0$ is large enough. 
\end{prop}
\begin{proof}
It is easy to check $\mathbf{w^{\pm}}(\cdot,0)=\mathbf{f}$ 
on $\R^n$. 
Note that 
\[
\frac{\pl\mathbf{g^{\ep}}}{\pl t}
=\frac{1}{\ep}(K-I)\mathbf{g^{\ep}}
\ \textrm{and} \ 
|D\mathbf{g}|\le Ce^{-\frac{\del t}{\ep}}. 
\]
Thus, we can check easily that 
$\mathbf{w^{\pm}}$ are a subsolution and a supersolution 
of \eqref{eq:m-system}, respectively,  if $C>0$ is large enough. 
\end{proof}

By a rather standard argument by using the perturbed test functions 
we can get Theorem \ref{thm:main-m} as in the proof of Theorem 
\ref{thm:main1}.


\section{Dirichlet Problems}\label{sec:dirichlet}
In this section  
we consider the asymptotic behavior,
as $\ep$ tends to $0$, of
the viscosity solutions 
$(u^\ep_1,u^\ep_2)$ of 
Dirichlet boundary problems for weakly coupled 
systems of Hamilton--Jacobi equations 
\begin{numcases}
{(\textrm{D}_\ep) \hspace{1cm}}
u^\ep_{1}+H_{1}(\dfrac{x}{\ep},Du^\ep_{1}) +\dfrac{1}{\ep}( u^\ep_{1}-u^\ep_{2}) = 0
& in $\Om$, \nonumber \\
u^\ep_{2}+H_{2}(\dfrac{x}{\ep},Du^\ep_{2}) + \dfrac{1}{\ep}(u^\ep_{2}-u^\ep_{1}) = 0
& in $\Om$, \nonumber \\
u^\ep_{i}(x)=g_{i}(x)
& on $\bO$,  
\nonumber
\end{numcases}
where 
$\Om$ is a bounded domain of $\R^{n}$ {with the Lipschitz boundary},  
the Hamiltonians $H_{i} \in C(\R^n \times \R^n)$ are 
assumed to satisfy (A1)-(A2) and 
$g_{i}\in C(\bO)$ are given functions for $i=1,2$.

Concerning the Dirichlet problem, 
classical works required continuous solutions up to the boundary to satisfy the prescribed data on the entire boundary. This can be achieved for special classes of equations by imposing compatibility conditions on the boundary data or by assuming the existence of appropriate super and subsolutions. 
However, in general, we cannot expect that there exists a (viscosity) solution satisfying the boundary condition in the classical sense. 
After Soner studied the state constraints problems in terms of PDE, the viscosity formulation for Dirichlet conditions was introduced by Barles and Perthame \cite{BP} and Ishii \cite{I}. 
In this paper we deal with solutions satisfying Dirichlet boundary conditions 
in the sense of viscosity solutions.

\begin{thm} \label{thm:dirichlet}
Let $(u^\ep_1, u^\ep_2)$ be the solution of $({\rm D}_{\ep})$. 
Then $u^\ep_i$ converge locally uniformly to the same limit $u$ 
on $\Om$ as $\ep \to 0$ for $i=1,2$ and $u$ solves
\begin{equation}\label{HJ.limit-D}
\begin{cases}
u+\ol{H}(Du)=0 \qquad &\textrm{in } \Om, \\
u = \ol{g} &\textrm{on } \bO, 
\end{cases}
\end{equation}
where $\ol{g}:=\min\{g_1,g_2\}$ on $\bO$. 
\end{thm}

\begin{lem}\label{lem:classic} 
If $(u_{1}^{\ep}, u_{2}^{\ep})$ 
is a bounded upper semicontinuous subsolution of $({\rm D}_{\ep})$, 
then $u_{i}^{\ep}(x)\le g_{i}(x)$ 
for all $x\in\bO$ and $i=1,2$. 
\end{lem}

\begin{proof}
Fix $x_0\in\bO$. 
Choose a sequence $\{x_k\}_{k\in\N}\subset\R^n\setminus\cO$ 
such that $|x_0-x_k|=1/k^2$. 
Define the functions $\phi_{1}:\cO\to\R$ by 
$\phi_{1}(x):=u_{1}^{\ep}(x)-k|x-x_k|$. 
Let $r>0$ and 
$\xi_k\in B(x_0,r)\cap\cO$ be a maximum point of 
$\phi_{1}$ on $B(x_0,r)\cap\cO$. 
Since $\phi_{1}(\xi_k)\ge \phi_{1}(x_0)$, we have 
$k|\xi_k-x_k|\le u_{1}^{\ep}(\xi_k)-u_{1}^{\ep}(x_0)+k|x_0-x_k|\le C$, 
where $C>0$ is a constant independent of $k$. 
Thus, $\xi_k\to x_0$ as $k\to\infty$. 
Moreover, noting that 
$u_{1}^{\ep}(x_0)\le\liminf_{k\to\infty}(u_{1}^{\ep}(\xi_k)+k|x_0-x_k|)
\le\limsup_{k\to\infty}u_{1}^{\ep}(\xi_k)+\limsup_{k\to\infty}k|x_0-x_k|
\le u_{1}^{\ep}(x_0)$, 
we get $u_{1}^{\ep}(\xi_k)\to u_{1}^{\ep}(x_0)$ as $k\to\infty$. 
By the viscosity property of $u_{1}^{\ep}$, we have 
\begin{align}
&u_{1}^{\ep}(\xi_k)+H_{1}(\frac{\xi_k}{\ep},p_k)
+\frac{1}{\ep}(u_{1}^{\ep}(\xi_k)-u_{2}(\xi_k))
\le0 \ \textrm{or}\label{pf:lem:classic-1}\\
&u_{1}^{\ep}(\xi_k)\le g_{1}(\xi_k), \nonumber
\end{align}
where
$p_k=k(\xi_k-x_k)/|\xi_k-x_k|$. 
Noting that $|p_k|=k$, 
by (A1), 
we see that the left-hand side of \eqref{pf:lem:classic-1} is positive 
for a sufficiently large $k\in\N$ and then 
we must have $u_{1}^{\ep}(\xi_k)\le g_{1}(\xi_k)$. 
Sending $k\to\infty$, 
we get $u_{1}^{\ep}(x_0)\le g_{1}(x_0)$. 
Similarly, we get $u_{2}^{\ep}(x_0)\le g_{2}(x_0)$ on $\bO$.  
\end{proof}

\begin{lem}\label{lem:equilip}
The families $\{u_{i}^{\ep}\}_{\ep>0}$ are equi-Lipschitz continuous in $\Om$ 
for $i=1,2$. 
\end{lem}

\begin{proof}
Set $M:=\max_{i=1,2} (\| H_{i}(\cdot,0) \|_{L^\infty(\R^n)}+\|g_i\|_{L^\infty(\partial \Omega)})$. 
Then $(-M,-M)$ and $(M,M)$ are a subsolution and a supersolution 
of $\D$, respectively. 
By the comparison principle for $\D$ we have $|u_{i}^{\ep}|\le M$. 
Adding two equations in $\D$ we get 
\[
u_1^{\ep}+u_2^{\ep}
+H_1(\frac{x}{\ep}, Du_1^{\ep})
+H_2(\frac{x}{\ep}, Du_2^{\ep})=0
\]
for almost every $x\in\Om$, which implies that 
$|Du_i^{\ep}|\le M^{'}$ in the sense of viscosity solutions 
for some $M^{'}>0$, which is independent of $\ep$. 
\end{proof}

\begin{proof}[Proof of Theorem {\rm \ref{thm:main1}}]
By Lemma \ref{lem:equilip}
we can extract a subsequence
$\{\ep_j\}$ converging to $0$ so that
$u_{i}^{\ep_j}$ 
 converges locally uniformly to $u_{i}\in C(\cO)$ 
for $i=1,2$.
By usual observations, we get that $u_{1}=u_{2}=:u$. 
Since \eqref{HJ.limit-D} has a unique solution,
it is enough for us to prove that 
$u$ is a solution of \eqref{HJ.limit-D}.

We only prove that $u$ is a supersolution of \eqref{HJ.limit-D}, 
since in view of Lemma \ref{lem:classic} we can easily see 
that $u$ is a subsolution of \eqref{HJ.limit-D}.

Let $\phi\in C^{1}(\cO)$ be a test function such that 
$u-\phi$ takes a strict minimum at $x_{0}\in\cO$. 
We only consider the case where $x_{0}\in\bO$, 
since we can prove by a similar way to the proof of 
Theorem \ref{thm:main1}
in the case where 
$x_{0}\in\Om$. 
It is enough for us to prove that
$u(x_0)+\ol{H}(D\phi(x_0))\ge0$
provided that
$(u-\ol{g})(x_0)<0$.

Let $(v_1, v_2)$ be a solution of $(\textrm{E}_P)$ with $P:=D\phi(x_{0})$. 
We consider 
\[
m^{\ep}:=\min_{i\in\{1,2\}}
\min_{x\in\cO}\bigl(u^{\ep}_{i}(x)-\phi(x)-\ep v_{i}(\frac{x}{\ep})\bigr). 
\]
Pick $i^{\ep}\in\{1,2\}$ and $x^{\ep}\in\cO$ so that 
$m^{\ep}=u^{\ep}_{i^{\ep}}(x^{\ep})-\phi(x^{\ep})-
\ep v_{i^{\ep}}(x^{\ep}/\ep)$. 
Also choose $j^{\ep}\in\{1,2\}$ such that 
$\{i^{\ep},j^{\ep}\}=\{1,2\}$. 
We only consider the case where $x^{\ep}\in\bO$ again.  
Since $u^\ep_{i^{\ep}}$ converges to $u$ locally uniformly on 
$\cO$, $\ep v_{i^\ep}(\cdot/\ep)$ converges to $0$ uniformly on $\cO$ 
as $\ep\to0$ and $u-\phi$ takes a strict maximum at $x_0$, 
we see that $x^{\ep}\to x_0$ as $\ep\to0$. 
Thus, if $\ep$ is small enough, then we may assume that 
$(u_{i^\ep}^\ep-g_{i^{\ep}})(x^{\ep})<0$.

For $\al>0$ we define the function $\Phi_{\al}:\cO\times\R^n\to\R$ by 
\[
\Phi_{\al}(x,y):=
u_{{i}^{\ep}}^{\ep}(x)-\phi(x)-\ep v_{i^{\ep}}\bigl(\frac{y}{\ep}\bigr)
+\frac{1}{2\al^{2}}|x-y|^{2}+\dfrac{1}{2}|x-x^{\ep}|^2. 
\]
Let $\Phi_{\al}$ achieve its minimum over $\cO\times\R^n$ at 
some $(x_{\al}^{\ep}, y_{\al}^{\ep})$. 
Since we may assume by taking a subsequence if necessary that 
$x_{\al}^{\ep}\to x^{\ep}$ 
as $\al\to0$, we have
\[
(u_{i^\ep}-g_{i^{\ep}})(x_{\al}^{\ep})<0 \ 
\text{for small} \ \al>0. 
\]
Therefore, by the definition of viscosity solutions, we have 
\[
u_{i^{\ep}}
+H_{i^{\ep}}(\frac{x_{\al}^{\ep}}{\ep},D\phi(x_{\al}^{\ep})-p_{\al}^{\ep}
-(x_{\al}^{\ep}-x^{\ep}))
+\frac{1}{\ep}(u^{\ep}_{i^{\ep}}-u^{\ep}_{j^{\ep}})(x_{\al}^{\ep})
\ge0,
\]
where $p_{\al}^{\ep}:=(x_{\al}^{\ep}-y_{\al}^{\ep})/\al^2$. 
Also, we have 
\[
H_{i}(\frac{y_{\al}^{\ep}}{\ep},P-p_{\al}^{\ep})
+(v_{i^{\ep}}-v_{j^{\ep}})(\frac{y_{\al}^{\ep}}{\ep})\le\ol{H}(P),  
\]
since $(v_1, v_2)$ is a solution of $(\textrm{E}_P)$.

A priori Lipschitz estimate implies 
$|p_{\al}^{\ep}|\le C$ 
for some $C>0$ which is independent of $\al$ and $\ep$.
Without loss of generality, we may assume that 
$p_{\al}^{\ep}\to p^{\ep}$ 
by taking a subsequence $\{ \al_j\}$ converging to $0$ if necessary. 
Send $\al\to0$ in 
the above inequalities to obtain
\begin{align*}
&u^{\ep}_{i^{\ep}}(x^{\ep})
+H_{i^{\ep}}(\frac{x^{\ep}}{\ep},D\phi(x^{\ep})-p^{\ep})
+\frac{1}{\ep}(u^{\ep}_{i^{\ep}}(x^{\ep})
-u^{\ep}_{j^{\ep}}(x^{\ep}))
\ge0, \\ 
&H_{i^{\ep}}(\frac{x^{\ep}}{\ep}, P-p^{\ep})
+v_{i^{\ep}}(\frac{x^{\ep}}{\ep})-v_{j^{\ep}}(\frac{x^{\ep}}{\ep})
\le\ol{H}(P).  
\end{align*}
Noting that 
$u^{\ep}_{i^{\ep}}(x^{\ep})-\phi(x^{\ep})-\ep v_{i^{\ep}}(x^{\ep}/\ep)
\le u^{\ep}_{j^{\ep}}(x^{\ep})-\phi(x^{\ep})-\ep v_{j^{\ep}}(x^{\ep}/\ep)$, 
we get that
\[
u^{\ep}_{i^{\ep}}(x^{\ep})+\ol{H}(P)
\ge 
H_{i^{\ep}}(\frac{x^{\ep}}{\ep},D\phi(x^{\ep})-p^{\ep})
-
H_{i^{\ep}}(\frac{x^{\ep}}{\ep}, P-p^{\ep})
\ge 
-\sig(|D\phi(x^{\ep})-P|),
\]
for some modulus $\sig$.
Sending $\ep\to0$ yields the conclusion. 
\end{proof}

In order to explain the relation between (D$_{\ep}$) and 
the exit-time problem in the optimal control theory, 
we assume that the Hamiltonians $H_i$ are convex in the $p$-variable 
henceforth. 
We next define the associated value functions, which
give us an intuition about the effective boundary datum
$\ol{g}$ in Theorem \ref{thm:dirichlet}.

For $\ep>0$ we define the functions $u_i^{\ep}: \cO\to\R$ by 
\begin{equation}\label{def:value-D}
u_{i}^{\ep}(x):=
\inf\Big\{\E_{i}\Big(\int_{0}^{\tau}e^{-s}L_{\nu^{\ep}(s)}(\frac{\eta(s)}{\ep},-\dot{\eta}(s))\,ds
+e^{-\tau}g_{\nu^{\ep}(\tau)}(\eta(\tau)) 
\Big)\Big\}, 
\end{equation}
where the infimum is taken over $\eta\in\AC([0,\infty),\cO)$ such that 
$\eta(0)=x$ and $\tau\in[0,\infty]$ such that $\eta(\tau)\in\bO$ 
and if $\tau=\infty$, then we set $e^{-\infty}:=0$. 
Here $\E_{i}$ denotes the expectation of 
a process with $\nu^{\ep}(0)=i$, where 
$\nu^{\ep}$  is a $\{1,2\}$-valued continuous-time Markov chain 
given by \eqref{markov}.

\begin{thm}\label{thm:sol-D}
{
Assume that the functions $u_i^{\ep}$ given by \eqref{def:value-D} 
are continuous on $\cO$. }
Then the pair $(u_1^{\ep},u_2^{\ep})$ is a solution of $({\rm D}_{\ep})$. 
\end{thm}

The proof of Theorem \ref{thm:sol-D} is given in Appendix.  
See \cite{BP,I} for single equations. 
The value functions defined by \eqref{def:value-D} give us 
an intuitive explanation of the reason why the boundary datum $\ol{g}$ 
of the limit solution $u$ is the minimum of 
$g_{i}$ for $i=1,2$.
If we send $\ep$ to $0$, then the switching rate becomes very fast
but it does not really affect the exit time as we can choose to stay in 
{$\cO$}
as long as we like.
And hence, we can control the exit state in such a way that the exit cost 
is the minimum of two given exit costs $g_i$.
On the other hand, when we consider the value function \eqref{def:value} 
associated with the initial value problem, 
we cannot control the terminal state and also the timing of jumps, 
which are only determined by a probabilistic way given by \eqref{markov}. 
This is the main difference between Dirichlet problems 
and initial value problems and the reason why  
the effective Dirichlet boundary value and the effective initial value 
are different.

\section{Appendix}

We now prove Theorems \ref{thm:sol}, and \ref{thm:sol-D} by basically using 
the dynamic programming principles, which are pretty standard in the theory of viscosity solutions.
{
Throughout this section we always assume in addition to (A1), (A2) 
that $p\mapsto H_{i}(\xi,p)$ 
are convex for $i=1,2$. 
}

\begin{thm}[Verification Theorem]\label{thm:sol}
{
Assume that the functions $u_i^{\ep}$ given by \eqref{def:value} 
are continuous on $\cQ$. }
Then the pair  $(u_1^{\ep},u_2^{\ep})$ is a solution of $({\rm C}_\ep)$. 
\end{thm}
Let $\ep=1$ for simplicity in what follows. 
By abuse of notations we write $(u_1,u_2)$ for $(u_1^1, u_2^1)$ and $\nu$ for $\nu^1$.

\begin{prop}[Dynamic Programming Principle]\label{prop:dpp}
For any $x\in\R^n$, $0\le h\le t$ and $i=1,2$ we have 
\begin{equation}\label{dpp}
u_{i}(x,t)=
\inf\Big\{\E_{i}\Big(\int_{0}^{h}L_{\nu(s)}(\eta(s),-\dot{\eta}(s))\,ds
+u_{\nu(h)}(\eta(h),t-h)\Big)\Big\}, 
\end{equation}
where 
the infimum is taken over $\eta\in\AC([0,h],\R^n)$ with $\eta(0)=x$. 
\end{prop}
\begin{proof}
We denote by $v_i(x,t;h)$ the right-hand side of \eqref{dpp}. 
Let $\eta$ be a trajectory in $\AC([0,t],\R^n)$ 
with $\eta(0)=x$ and $\nu$ be a process with 
$\nu(0)=i$ which satisfies \eqref{markov}. 
Set $\tilde{\eta}(s):=\eta(s+h)$ and 
$\tilde{\nu}(s):=\nu(s+h)$ 
for $s \in [0,t-h]$.
We have 
\begin{align*}
&\E_{i}\Big(\int_{0}^{t}L_{\nu(s)}(\eta(s),-\dot{\eta}(s))\,ds
+f_{\nu(t)}(\eta(t))\Big)\\
=&\,
\E_{i}\Big(\int_{0}^{h}L_{\nu(s)}(\eta(s),-\dot{\eta}(s))\,ds
+\int_{h}^{t}L_{\nu(s)}(\eta(s),-\dot{\eta}(s))\,ds
+f_{\nu(t)}(\eta(t))\Big)\\ 
=&\,
\E_{i}\Big(\int_{0}^{h}L_{\nu(s)}(\eta(s),-\dot{\eta}(s))\,ds\Big)
+
\E_{\nu(h)}\Big(\int_{0}^{t-h}L_{\tilde{\nu}(s)}(\tilde{\eta}(s),-\dot{\tilde{\eta}}(s))\,ds
+f_{\tilde{\nu}(t-h)}(\tilde{\eta}(t-h)) \Big)\\ 
\ge&\,
\E_{i}\Big(\int_{0}^{h}L_{\nu(s)}(\eta(s),-\dot{\eta}(s))\,ds
+u_{\nu(h)}(\eta(h),t-h)\Big)\\
\ge&\,
v_{i}(x,t;h),  
\end{align*}
in view of the memoryless property of $\nu$, 
which implies $u_i(x,t)\ge v_{i}(x,t;h)$.

Let $\del_1\in\AC([0,h],\R^n)$ and $\del_2\in\AC([0,t-h],\R^n)$
be trajectories with $\del_{1}(h)=\del_{2}(0)$ and $\del_{1}(0)=x$. 
Set 
\[
\eta(s):=
\left\{
\begin{array}{ll}
\del_{1}(s) & \textrm{for all} \ s\in[0,h], \\
\del_{2}(s-h) & \textrm{for all} \ s\in[h,t].  \\
\end{array}
\right. 
\]
Let $\nu$ be a process with $\nu(0)=i$ which satisfies \eqref{markov}. 
Note that 
\begin{align*}
&
\int_{0}^{h}L_{\nu(s)}(\del_{1}(s),-\dot{\del_{1}}(s))\,ds
+\int_{0}^{t-h}L_{\nu(s+h)}(\del_{2}(s),-\dot{\del_{2}}(s))\,ds
+f_{\nu(t)}(\del_{2}(t-h))\\ 
=&\, 
\int_{0}^{t}L_{\nu(s)}(\eta(s),-\dot{\eta}(s))\,ds
+f_{\nu(t)}(\eta(t)). 
\end{align*}
We have 
\begin{align*}
&\E_{i}\Big(\int_{0}^{h}L_{\nu(s)}(\del_{1}(s),-\dot{\del_{1}}(s))\,ds
+\int_{0}^{t-h}L_{\nu(s+h)}(\del_{2}(s),-\dot{\del_{2}}(s))\,ds
+f_{\nu(t)}(\del_{2}(t-h))\Big)\\ 
=&\, 
\E_{i}\Big(\int_{0}^{t}L_{\nu(s)}(\eta(s),-\dot{\eta}(s))\,ds
+f_{\nu(t)}(\eta(t))\Big)\\
\ge&\, 
u_i(x,t). 
\end{align*}
Take the infimum on all admissible 
$\del_2$ to obtain  
\[
\E_{i}\Big(
\int_{0}^{h}L_{\nu(s)}(\del_{1}(s),-\dot \del_{1}(s))\,ds
+u_{\nu(h)}(\del_{1}(h),t-h)\Big)
\ge 
u_i(x,t), 
\]
which implies $v_{i}(x,t;h)\ge u_i(x,t)$.
\end{proof}

\begin{proof}[Proof of Theorem {\rm \ref{thm:sol}}]
It is obvious to see that $(u_1,u_2)(\cdot,0)=(f_1,f_2)$ on $\R^n$. 
We first prove that $u_1$ is a subsolution of  $(\textrm{C}_1)$. 
We choose a function $\phi\in C^{1}(\Q)$ such that 
$u_1-\phi$ has a strict maximum at $(x_0,t_0) \in \Q$ and 
$(u_1-\phi)(x_0,t_0)=0$.

Let $h>0$.  
By Proposition \ref{dpp} we have 
\begin{equation}\label{dpp-sub}
u_1(x_0,t_0)\le
\E_{i}\Big(\int_{0}^{h}L_{\nu(s)}(\eta(s),-\dot{\eta}(s))\,ds
+u_{\nu(h)}(\eta(h),t_0-h)\Big) 
\end{equation}
for any $\eta\in\AC([0,h],\R^n)$ with $\eta(0)=x_0\in\R^n$ 
and $\dot{\eta}(0)=q\in\R^n$. 
Since $\nu$ is a continuous-time Markov chain which satisfies \eqref{markov},  
the probability that $\nu(h)=2$ is $c_1h+o(h)$ 
and the probability that $\nu(h)=1$ is $1-(c_1h+o(h))$. 
By \eqref{dpp-sub} we obtain 
\begin{align*}
&\phi(x_0,t_0)=
u_1(x_0,t_0)\\
\le&\, 
(1-c_1h-o(h))
\Big(\int_{0}^{h}L_{1}(\eta,-\dot{\eta})\,ds+u_1(\eta(h),t_0-h)\Big) \\
&\,+
(c_1h+o(h))
\Big(\int_{0}^{h}L_{2}(\eta,-\dot{\eta})\,ds+u_2(\eta(h),t_0-h)\Big)
{+o(h)}\\
\le&\, 
\int_{0}^{h}L_{1}(\eta,-\dot{\eta})\,ds+\phi(\eta(h),t_0-h)
{+o(h)}\\ 
&\, 
+
(c_1h+o(h))\Big(
\int_{0}^{h}L_{2}(\eta,-\dot{\eta})\,ds
+u_2(\eta(h),t_0-h)
-\int_{0}^{h}L_{1}(\eta,-\dot{\eta})\,ds
-u_1(\eta(h),t_0-h)
\Big).
\end{align*}
Thus, 
\begin{align*}
& 
\frac{\phi(\eta(0),t_0)-\phi(\eta(h),t_0-h)}{h}\\
\le&\, 
\frac{1}{h}\int_{0}^{h}L_{1}(\eta,-\dot{\eta})\,ds
{+\frac{o(h)}{h}}
+
(c_{1}+\frac{o(h)}{h})(u_2(\eta(h),t_0-h)-u_1(\eta(h),t_0-h))\\
&\,
+
(c_1+\frac{o(h)}{h})\Big(
\int_{0}^{h}L_{2}(\eta,-\dot{\eta})\,ds
-\int_{0}^{h}L_{1}(\eta,-\dot{\eta})\,ds
\Big). 
\end{align*}
Sending $h\to0$,  
we obtain 
\[
\phi_t(x_0,t_0)+D\phi(x_0,t_0)\cdot (-q)\le L_1(x_0,-q)
+c_1(u_2-u_1)(x_0,t_0) 
\ \text{for all} \ q\in\R^n,  
\]
which implies 
$\phi_t(x_0,t_0)+H_1(x_0,D\phi(x_0,t_0))+c_1(u_1-u_2)(x_0,t_0)\le0$.

Next we prove that $u_1$ is a supersolution of $(\textrm{C}_1)$.
We choose a function $\phi\in C^{1}(\Q)$ such that 
$u_1-\phi$ has a strict minimum at $(x_0,t_0)\in \Q$ and 
$(u_1-\phi)(x_0,t_0)=0$. Take $h, \del>0$.  
By Proposition \ref{dpp} we have 
\begin{equation}\label{dpp-sup}
u_1(x_0,t_0)+\del>
\E_{1}\Big(\int_{0}^{h}L_{\nu(s)}(\eta_{\del}(s),-\dot{\eta}_{\del}(s))\,ds
+u_{\nu(h)}(\eta_{\del}(h),t_0-h)\Big) 
\end{equation}
for some $\eta_{\del}\in\AC([0,h],\R^n)$ with $\eta_{\del}(0)=x_0$. 
Since $\nu$ is a continuous-time Markov chain which satisfies \eqref{markov},  
by a similar calculation to the above we obtain 
\begin{align*}
&\phi(x_0,t_0)+\del=
u_1(x_0,t_0)+\del\\
>&\, 
\int_{0}^{h}L_{1}(\eta,-\dot{\eta})\,ds+\phi(\eta(h),t_0-h)
{+o(h)}\\ 
&\, 
+
(c_1h+o(h))\Big(
\int_{0}^{h}L_{2}(\eta,-\dot{\eta})\,ds+u_2(\eta(h),t_0-h)
-\int_{0}^{h}L_{1}(\eta,-\dot{\eta})\,ds-u_1(\eta(h),t_0-h)
\Big).
\end{align*}
Thus, 
\begin{align*}
\frac{\del}{h}
>&\, 
\frac{1}{h}\int_{0}^{h}
\dfrac{d\phi(\eta_{\del}(s),t_0-s)}{ds}+
L_{1}(\eta_{\del},-\dot{\eta_{\del}})\,ds
{+\frac{o(h)}{h}}
\\
&\, 
+
(c_1+\frac{o(h)}{h})(u_2(\eta_{\del}(h),t_0-h)-u_1(\eta_{\del}(h),t_0-h))\\
&\, 
+
(c_1+\frac{o(h)}{h})\Big(
\int_{0}^{h}L_{2}(\eta_{\del},-\dot{\eta_{\del}})\,ds
-\int_{0}^{h}L_{1}(\eta_{\del},-\dot{\eta_{\del}})\,ds
\Big)\\
=&\, 
\frac{1}{h}\int_{0}^{h}
-\phi_t(\eta_{\del}(s),t_0-s)-D\phi\cdot(-\dot{\eta_{\del}}(s))
+
L_{1}(\eta_{\del},-\dot{\eta_{\del}})\,ds\\
&\, 
+
(c_1+\frac{o(h)}{h})(u_2(\eta_{\del}(h),t_0-h)-u_1(\eta_{\del}(h),t_0-h))
+ O(h) \\
\ge &\, 
\frac{1}{h}\int_{0}^{h}
-\big(\phi_t(\eta_{\del}(s),t_0-s)+H_{1}(\eta_{\del}(s),D\phi)\big)\,ds\\
&\, 
+
(c_1+\frac{o(h)}{h})(u_2(\eta_{\del}(h),t_0-h)-u_1(\eta_{\del}(h),t_0-h))
+O(h). 
\end{align*}
We finally set $\del=h^2$ and  let $h\to0$ to yield the conclusion. 
\end{proof}

{By a similar argument to the proof of Proposition \ref{prop:dpp} 
we can prove  
}
\begin{prop}[Dynamic Programming Principle for {\rm\eqref{def:value-D}}]\label{prop:dpp-D}
For any $x\in\R^n$, $h\ge0$ and $i=1,2$ we have 
\begin{multline}\label{dpp-D}
u_{i}^\ep(x)
= 
\inf\Big\{\E_{i}\Big(\int_{0}^{h\land\tau}
e^{-s}L_{\nu^\ep(s)}(\frac{\eta(s)}{\ep},-\dot{\eta}(s))\,ds\\
+\bone_{\{h<\tau\}}e^{-h}u_{\nu^\ep(h)}(\eta(h))
+\bone_{\{h\ge\tau\}}e^{-\tau}g_{\nu^\ep(\tau)}(\eta(\tau)) 
\Big)\Big\}, 
\end{multline}
where 
$\nu^\ep$ with $\nu^\ep(0)=i$ is a $\{1,2\}$-valued continuous-time 
Markov chain which satisfies \eqref{markov} and  
the infimum is taken over $\eta\in\AC([0,h],\cO)$ such that 
$\eta(0)=x$ and $\tau\in[0,h]$ such that 
$\eta(\tau)\in\bO$. 
\end{prop}

\begin{proof}[Proof of Theorem {\rm\ref{thm:sol-D}}]
As above, set $\ep=1$.
We only prove in what follows 
that $u_{i}$ satisfy the Dirichlet boundary condition 
in the sense of viscosity solutions, as we can 
prove $u_{i}$ satisfy the equations by an argument similar 
to the proof of Theorem \ref{thm:sol}.  
Since it is clear to see that $u_{i}\le g_i$ on $\bO$ in the classical 
sense from the definition of $u_{i}$,  
we only need to prove that $(u_1, u_2)$ is a supersolution 
of $(\textrm{D}_{1})$ and particularly that $u_1$ satisfies the boundary 
condition in the viscosity solution sense. 
Take $x_0\in\bO$ so that 
\begin{equation}\label{pf:d-1}
(u_1-g_1)(x_0)<0, 
\end{equation}
and $\phi\in C^1(\cO)$ satisfying 
$(u_1-\phi)(x_0)=\min_{\cO}(u_1-\phi)=0$. 
By Proposition \ref{prop:dpp-D}
we have 
\begin{align*}
&u_1(x_0)+h^2\\
>&\, 
\E_{1}\Big(\int_{0}^{h\land\tau_{h}}
e^{-s}L_{\nu(s)}(\eta_{h}(s),-\dot{\eta}_{h}(s))\,ds
+
\bone_{\{h<\tau_{h}\}}e^{-h}u_{\nu(h)}(\eta_{h}(h)) 
+
\bone_{\{h\ge\tau_{h}\}}e^{-\tau_{h}}g_{\nu(\tau_{h})}(\eta_{h}(\tau_{h}))
\Big)\\
\ge&\, 
\E_{1}\Big(\int_{0}^{h\land\tau_{h}}
e^{-s}L_{\nu(s)}({\eta_{h}(s)},-\dot{\eta}_{h}(s))\,ds
+
e^{-(h\land\tau_{h})}u_{\nu(h\land\tau_{h})}(\eta_{h}(h\land\tau_{h}))
\Big)
\end{align*}
for some $\eta_{h}\in\AC([0,h],\cO)$ such that 
$\eta_h(0)=x_0$ and $\tau_{h}\in[0,h]$. 
In view of \eqref{pf:d-1}, we have 
$\tau_{h}>0$ for small $h>0$. 
Therefore by a similar calculation as in the proof of Theorem \ref{thm:sol} we get 
\[
u_{1}(x_0)+H_1(x_0,D\phi(x_0))+c_1(u_1-u_2)(x_0) \ge0. 
\quad\qedhere\]
\end{proof}


%

\bibliographystyle{amsplain}

\begin{thebibliography}{1}

\bibitem{AS1}
S. N. Armstrong, P. E. Souganidis, 
\emph{Stochastic homogenization of {H}amilton--{J}acobi and degenerate Bellman equations in unbounded environments}, 
J. Math. Pures Appl. (9), {\bf 97} (2012), no. 5, 460--504.

\bibitem{AS2}
S. N. Armstrong, P. E. Souganidis, 
\emph{Concentration phenomena for neutronic multigroup diffusion in random environments}, 
Ann. Inst. H. Poincar\'e Anal. Non Lin\'eaire, {\bf 30} (2013), 419--439.

\bibitem{BP} 
G. Barles, B. Perthame, 
\emph{Exit time problems in optimal control and vanishing viscosity method}, 
SIAM J. Control Optim. {\bf 26} (1988), no. 5, 1133--1148.

\bibitem{BS}
{
J. Busca, B. Sirakov, 
\emph{Harnack type estimates for nonlinear elliptic systems and applications}, 
Ann. Inst. H. Poincare Anal. Non Lineaire 21 (2004), no. 5, 543--590.
}

\bibitem{CGT2}
F. Cagnetti, D. Gomes, H. V. Tran, 
\emph{Adjoint methods for obstacle problems and weakly coupled systems of {P}{D}{E}},
ESAIM: Control, Optimisation and Calculus of Variations {\bf 19} (2013), no. 3, 754--779. 


\bibitem{CCM}
F. Camilli, A. Cesaroni, C. Marchi,
\emph{Homogenization and vanishing viscosity in fully nonlinear elliptic equations: rate of convergence estimates},
Adv. Nonlinear Stud., {\bf 11} (2011), no. 2, 405--428.

\bibitem{CLL}
F. Camilli, O. Ley, P. Loreti, 
\emph{Homogenization of monotone systems of Hamilton--Jacobi equations}, 
ESAIM Control Optim. Calc. Var. {\bf 16} (2010), no. 1, 58--76. 

\bibitem{CLLN}
F. Camilli, O. Ley, P. Loreti, V. Nguyen, 
\emph{Large time behavior of weakly coupled systems of first-order Hamilton-Jacobi equations}, 
NoDEA Nonlinear Differential Equations Appl. {\bf 19} (2012), 
no. 6, 719--749.

\bibitem{CM}
F. Camilli, C. Marchi, 
\emph{Continuous dependence estimates and homogenization of quasi-monotone systems of fully nonlinear second order parabolic equations}, 
Nonlinear Analysis TMA {\bf 75} (2012), 5103--5118.  

\bibitem{CDI} 
I. Capuzzo-Dolcetta, H. Ishii, 
\emph{On the rate of convergence in homogenization of Hamilton--Jacobi equations}, 
Indiana Univ. Math. J. {\bf 50} (2001), no. 3, 1113--1129. 

\bibitem{CIPP}
G. Contreras, R. Iturriaga, G. P. Paternain, M. Paternain,
\emph{Lagrangian graphs, minimizing measures and {M}a\~n\'e's critical values},
Geom. Funct. Anal. {\bf 8} (1998), no. 5, 788--809.

\bibitem{C1}
M. C. Concordel, 
\emph{Periodic homogenization of Hamilton--Jacobi equations: additive eigenvalues and variational formula}, 
Indiana Univ. Math. J. {\bf 45} (1996), no. 4, 1095--1117. 

\bibitem{C2}
M. C. Concordel, 
\emph{Periodic homogenisation of Hamilton--Jacobi equations. {II}. Eikonal equations}, 
Proc. Roy. Soc. Edinburgh Sect. A {\bf 127} (1997), no. 4, 665--689. 

\bibitem{D}
M. H. A. Davis, 
\emph{Piecewise-deterministic Markov processes: a general class of nondiffusion stochastic models}, 
J. Roy. Statist. Soc. Ser. B {\bf 46} (1984), no. 3, 353--388.

\bibitem{EF}
A. Eizenberg, M. Freidlin, 
\emph{On the Dirichlet problem for a class of second order PDE systems with
small parameter}, 
Stochastics Stochastics Rep., {\bf 33} (3-4):111--148, 1990. 

\bibitem{EL}
H. Engler, S. M. Lenhart, 
\emph{Viscosity solutions for weakly coupled systems of Hamilton--Jacobi equations}, 
Proc. London Math. Soc. (3) {\bf 63} (1991), no. 1, 212--240. 

\bibitem{E}
L. C. Evans, 
\emph{The perturbed test function method for viscosity solutions of nonlinear {PDE}}, 
Proc. Roy. Soc. Edinburgh Sect. A {\bf 111} (1989), no. 3-4, 359--375. 

\bibitem{E2}
L. C. Evans, 
\emph{Periodic homogenisation of certain fully nonlinear partial differential equations}, 
Proc. Roy. Soc. Edinburgh Sect. A {\bf 120} (1992), no. 3-4, 245--265. 

\bibitem{EG}
L. C. Evans, D. Gomes, 
\emph{Effective Hamiltonians and averaging for Hamiltonian dynamics. I}, 
Arch. Ration. Mech. Anal. {\bf 157} (2001), no. 1, 1--33.  


\bibitem{Fe}
B. Fehrman, 
\emph{Stochastic homogenization of monotone systems of viscous
Hamilton--Jacobi equations with convex nonlinearities},
submitted.

\bibitem{FS}
{
W. H. Fleming, H. M. Soner, 
\emph{Controlled Markov processes and viscosity solutions}, 
Stochastic Modelling and Applied Probability, {\bf 25}. 
Springer, New York, 2006.
}

\bibitem{Gomes1}
D. Gomes,
\emph{A stochastic analogue of {A}ubry-{M}ather theory},
Nonlinearity {\bf 15}, (2002), no. 3, 581--603.

\bibitem{HI}
K. Horie, H. Ishii, 
\emph{Homogenization of Hamilton--Jacobi equations on domains with small scale periodic structure}, 
Indiana Univ. Math. J. {\bf 47} (1998), no. 3, 1011--1058. 

\bibitem{I}
H. Ishii, 
\emph{A boundary value problem of the Dirichlet type for Hamilton--Jacobi equations}, 
Ann. Scuola Norm. Sup. Pisa Cl. Sci. (4) {\bf 16} (1989), no. 1, 105--135. 

\bibitem{IK}
H. Ishii, S. Koike, 
\emph{Viscosity solutions for monotone systems of second-order elliptic {PDE}s}, 
Comm. Partial Differential Equations {\bf 16} (1991), no. 6-7, 1095--1128. 

\bibitem{IS}
H. Ishii, K. Shimano, 
\emph{Asymptotic analysis for a class of infinite systems of first-order {PDE}: nonlinear parabolic {PDE} in the singular limit}, 
Comm. Partial Differential Equations {\bf 28} (2003), no. 1-2, 409--438. 

\bibitem{KRV}
E. Kosygina, F. Rezakhanlou, S. R. S. Varadhan, 
\emph{Stochastic homogenization of Hamilton--Jacobi--Bellman equations}, 
Comm. Pure Appl. Math. {\bf 59} (10) (2006) 1489--1521.


\bibitem{LY}
S. M. Lenhart, N. Yamada, 
\emph{Viscosity solutions associated with switching game for piecewise-
deterministic processes}, 
Stochastics Stochastics Rep., {\bf38} (1): 27--47, 1992.





\bibitem{LPV}  
P.-L. Lions, G. Papanicolaou, S. R. S. Varadhan,  
\emph{Homogenization of Hamilton--Jacobi equations}, unpublished work (1987). 

\bibitem{LS1}
P.-L. Lions, P. E. Souganidis, 
\emph{Correctors for the homogenization of Hamilton--Jacobi equations in the stationary ergodic setting},
Comm. Pure Appl. Math. {\bf 56} (10) (2003) 1501--1524.


\bibitem{LS3}
P.-L. Lions, P. E. Souganidis, 
\emph{Stochastic homogenization of Hamilton--Jacobi and ``viscous" Hamilton--Jacobi equations with convex 
nonlinearities--revisited}, 
Commun. Math. Sci. {\bf 8} (2) (2010) 627--637.


\bibitem{MT1}
H. Mitake, H. V. Tran,
\emph{Remarks on the large time behavior of viscosity solutions of quasi-monotone weakly coupled systems of Hamilton--Jacobi equations}, Asymptot. Anal., {\bf 77} (2012), 43--70. 


\bibitem{RT1}
F. Rezakhanlou, J. E. Tarver, 
\emph{Homogenization for stochastic Hamilton--Jacobi equations}, 
Arch. Ration. Mech. Anal. {\bf 151} (4) (2000) 277--309.

\bibitem{Sch1}
R. W. Schwab, 
\emph{Stochastic homogenization of Hamilton--Jacobi equations in stationary ergodic spatio-temporal media}, 
Indiana Univ. Math. J. {\bf 58} (2) (2009) 537--581.

\bibitem{S}
K. Shimano, 
\emph{Homogenization and penalization of functional first-order {PDE}}, 
NoDEA Nonlinear Differential Equations Appl. {\bf 13} (2006), no. 1, 1--21. 

\bibitem{Sou1}
P. E. Souganidis, 
\emph{Stochastic homogenization of Hamilton--Jacobi equations and some applications}, 
Asymptot. Anal. {\bf 20} (1) (1999) 1--11.

\bibitem{Tr1}
H. V. Tran,
\emph{Adjoint methods for static Hamilton-Jacobi equations},
Calc. Var. Partial Differential Equations {\bf 41} (2011), no. 3-4, 301--319. 


\end{thebibliography}
\providecommand{\bysame}{\leavevmode\hbox to3em{\hrulefill}\thinspace}
\providecommand{\MR}{\relax\ifhmode\unskip\space\fi MR }
\providecommand{\MRhref}[2]{%
  \href{http://www.ams.org/mathscinet-getitem?mr=#1}{#2}
}
\providecommand{\href}[2]{#2}

\end{document}